\documentclass[a4paper,leqno]{amsart}

\usepackage{amsmath,amssymb,amsthm}
\usepackage[colorlinks=true]{hyperref}

\usepackage{bm}

\usepackage{algorithm}
\usepackage{algpseudocode}

\usepackage[shortlabels]{enumitem}

\def\C{{\mathbf C}}% complex numbers
\def\R{{\mathbf R}}% real numbers
\def\N{{\mathbf N}}% nonnegative integers
\def\F{{\mathcal F}}% Fourier

\def\conj{\mathcal C}

\def\virgp{\raise 2pt\hbox{,}}
\def\({\left(}
\def\){\right)}
\def\<{\left\langle}
\def\>{\right\rangle}
\def\le{\leqslant}
\def\ge{\geqslant}

\DeclareMathOperator{\RE}{Re}
\DeclareMathOperator{\IM}{Im}

\def\Eq#1#2{\mathop{\sim}\limits_{#1\rightarrow#2}}
\def\Tend#1#2{\mathop{\longrightarrow}\limits_{#1\rightarrow#2}}

\def\d{{\partial}}

\def\eps{\varepsilon}

\def\si{{\sigma}}
\def\O{{\mathcal O}}

\usepackage{graphicx} % For figures
\usepackage{subcaption} % For subfigure captions

\usepackage{xcolor} % Package to change textcolour

\definecolor{highlightred}{RGB}{156,18,8} % For custom inline comments

\usepackage{appendix} % For appendices

\theoremstyle{plain}
\newtheorem{theorem}{Theorem}[section]
\newtheorem{lemma}[theorem]{Lemma}
\newtheorem{corollary}[theorem]{Corollary}
\newtheorem{proposition}[theorem]{Proposition}

\theoremstyle{definition}

\newtheorem{remark}[theorem]{Remark}

\numberwithin{equation}{section}

\begin{document}

\title[Rigidity properties and numerics for NLS scattering]{On scattering for NLS:
  rigidity properties and numerical simulations via the lens transform}

\author[R. Carles]{R\'emi Carles}
\address[R. Carles]{CNRS \& Univ. Rennes\\ IRMAR - UMR 6625\\ 
  Rennes, France}
\email{Remi.Carles@math.cnrs.fr}

\author[G. Maierhofer]{Georg Maierhofer}
\address[G. Maierhofer]{Department of Applied Mathematics and Theoretical Physics, University of Cambridge,
  United Kingdom} 
\email{gam37@cam.ac.uk}

\begin{abstract}
  We analyse the scattering operator associated with the
    defocusing nonlinear Schr\"odinger equation which captures the
    evolution of solutions over an infinite time-interval under the
    nonlinear flow of this equation. The asymptotic nature of the
    scattering operator (involving unbounded time) makes its
    computation particularly challenging. We overcome this by
    exploiting the space–time compactification provided by the lens
    transform, marking the first use of this technique in numerical
    simulations. This results in a highly efficient and reliable
    methodology for computing the scattering operator in various
    regimes. In developing this approach we  introduce and prove
    several new identities and theoretical properties of the
    scattering operator. We support our construction with several
    numerical experiments which we show to agree with known analytical
    properties of the scattering operator,  and also address the case
    of long-range scattering for the one-dimensional cubic
    Schr\"odinger equation. Our simulations permit us to further
    explore regimes beyond current analytical understanding, and lead
    us to formulate new conjectures concerning fixed and rotating
    points of the operator, as well as its existence in the long-range
    setting for both defocusing and focusing cases. 
\end{abstract}

% OLD ABSTRACT BELOW
%\begin{abstract}
%	We analyse the scattering operator associated with the defocusing
%	nonlinear Schr\"odinger equation. We recall some properties known
%	for this operator, and present two new identities. Relying on the
%	lens transform, we perform numerical experiments, agreeing with
%	known analytic properties, and suggesting conjectures regarding
%	some open problems. Our simulations also treat the case of long
%	range scattering for the one-dimensional cubic Schr\"odinger
%	equation.   
%\end{abstract}
 
\thanks{This work was supported by Centre Henri Lebesgue,
program ANR-11-LABX-0020-0. A CC-BY public copyright license has
  been applied by the author to the present document and will be
  applied to all subsequent versions up to the Author Accepted
  Manuscript arising from this submission.}

\maketitle

\section{Introduction}
\label{sec:intro}

We consider  the defocusing nonlinear Schr\"odinger
equation 
\begin{equation}
  \label{eq:nls}
  i\d_t u + \frac{1}{2}\Delta u = \lvert u\rvert^{2\si} u,\quad
  (t,x)\in \R\times \R^d,\ d\ge 1,
\end{equation}
and assume that the nonlinearity is
energy-subcritical, $0<\si<\frac{2}{(d-2)_+}$. As the nonlinearity is
defocusing, global existence of solutions in the Sobolev space
$H^1(\R^d)$ is classical for such powers, as proven initially in
\cite{GV79Cauchy}. Once global existence is established, the next
standard question is the large time behaviour of the solution. As the
Schr\"odinger equation is a dispersive equation, it is likely that
the nonlinearity becomes small as time gets large, and, possibly,
eventually negligible. This is the underlying idea of scattering
theory, which we now summarise in a more quantitative way. 
\smallbreak

To understand the large time behaviour of solutions to \eqref{eq:nls},
two types of initial data are of special interest: 
\begin{align}
&\text{Initial value problem: } u_{\mid t=0} = u_0,\label{eq:CIcauchy}\\
  &\text{Final value problem: }U_0(-t)u(t)\big|_{t=-
  \infty}=u_-,\quad \text{where
  }U_0(t)=e^{i\frac{t}{2}\Delta}.\label{eq:CIasym} 
\end{align}
Asymptotic states $u_+$ for $t=+\infty$ can also be considered, we address
only the case $t=-\infty$ to avoid any ambiguity in the discussion.
The scattering operator can be defined after two steps,
which we describe below.\\

\noindent {\bf Existence of wave operators.} For this step, we impose
the behaviour of the nonlinear solution as time goes to (plus or minus)
infinity. More precisely, for a given asymptotic state $u_-$, we want
to find a solution to \eqref{eq:nls} such that
\begin{equation*}
  u(t)\Eq t {-\infty} U_0(t)u_- = e^{i\frac{t}{2}\Delta}u_-.
\end{equation*}
As the Schr\"odinger group is dispersive, $U_0(t)u_-$ goes to zero
pointwise, this is why the precise requirement is \eqref{eq:CIasym}. 
If we assume
\begin{equation}\label{eq:wave-op}
  \si>1\text{ if }d=1,\quad \si>\frac{2}{d+2}\text{ if }d\ge 2,
\end{equation}
then it is well known that for a given asymptotic state $u_-\in \Sigma=H^1\cap{\mathcal F}(H^1)$, 
where
\begin{equation}\label{eq:fourier}
  {\mathcal F} f(\xi)=\widehat
  f(\xi)=\frac{1}{(2\pi)^{d/2}}\int_{\R^d}f(x)e^{-ix\cdot \xi} dx,
\end{equation}
\eqref{eq:nls}--\eqref{eq:CIasym} has a unique, global, solution
$u\in C(\R;\Sigma)$ (\cite{GV79Scatt}, see also
\cite{CazCourant}). Its initial value $u_{\mid t=0}$ is the image of 
the asymptotic state under the action of the wave operator:
\begin{equation*}
  u_{\mid t=0}= W_-u_-. 
\end{equation*}
The space $\Sigma$ is also characterised as a
weighed Sobolev space, since
\begin{equation*}
  \Sigma = \{f\in L^2(\R^d),\quad \nabla f\in L^2(\R^d),\ x\mapsto |x|
  f(x)\in L^2(\R^d)\}.
\end{equation*}
We can define the norm on $\Sigma$ as
\begin{equation*}
  \|f\|_\Sigma = \|f\|_{L^2(\R^d)}+ \|\nabla f\|_{L^2(\R^d)}+
  \||x|f\|_{L^2(\R^d)}. 
\end{equation*}
\noindent {\bf Asymptotic completeness.}
We now turn to the large time behaviour of solutions to the Cauchy
problem. If 
\begin{equation}\label{eq:completeness}
  \si \ge \frac{2-d+\sqrt{d^2+12d+4}}{4d}=:\si_0(d),
\end{equation}
{where $\si_0(d)$ is called the Strauss exponent},
and $u_0\in \Sigma$, \eqref{eq:nls}--\eqref{eq:CIcauchy}
possesses asymptotic states:
\begin{equation}\label{eq:cvSigma}
  \exists u_\pm\in \Sigma,\quad
  \|U_0(-t)u(t)-u_\pm\|_\Sigma \Tend t {\pm \infty} 0: \quad 
  u_\pm=W_\pm^{-1}u_0.
\end{equation}
See \cite{GV79Scatt} for the case $\si>\si_0(d)$, and
\cite{CW92,NakanishiOzawa} for the case $\si=\si_0(d)$. 
The same is true under the weaker assumption \eqref{eq:wave-op},
provided that $u_0$ belongs to a sufficiently small ball in $\Sigma$
(whose size is not explicit). We note the inequalities
$1/d<\si_0(d)<2/d$, the value $\si=2/d$ corresponding to the
$L^2$-critical, or conformally invariant, nonlinearity. 
\smallbreak

The scattering operator associated to \eqref{eq:nls} is classically
defined as
\begin{equation}\label{eq:scattering_operator}
  S=W_+^{-1}\circ W_-: u_-\mapsto u_+.
\end{equation}
It maps $\Sigma$ to $\Sigma$ under the assumption
\eqref{eq:completeness}, and a small ball of $\Sigma$ to (a small ball
of) $\Sigma$ under the weaker assumption \eqref{eq:wave-op} (see
e.g. \cite[Theorem~7.5.6]{CazCourant}). Also, for $u_0\in 
\Sigma$ and $\si>1/d$, there exist $u_\pm \in H^1(\R^d)$ such that
\begin{equation}\label{eq:cvH1}
  \|u(t)-U_0(t)u_\pm\|_{H^1}\Tend t {\pm \infty}0,
\end{equation}
as established in \cite{BGTV23}, extending the initial result of
\cite{TsutsumiYajima}.

The assumption $\si>1/d$ is necessary in the sense that if $\si\le
1/d$, a long range theory is necessary as the above convergence is
possible only in the trivial case $u\equiv 0$. This follows from the
stronger result from \cite{Barab}: if $\si\le 1/d$ and
\begin{equation*}
  \|u(t)-U_0(t)u_+\|_{L^2(\R^d)}\Tend t {+ \infty}0,
\end{equation*}
then necessarily, $u=u_+\equiv 0$. Note that in the range
$1/d<\si<\si_0(d)$, it is not known whether scattering holds in
$\Sigma$, that is if \eqref{eq:cvH1} can be improved to
\eqref{eq:cvSigma} for large data $u_0\in \Sigma$. This point is
addressed numerically in Section~\ref{sec:short}. Our simulations
suggest that \eqref{eq:cvSigma} may not be true, at least for
some $\si\in(1/d,\sigma_0(d))$.  
\smallbreak

Apart from the existence of the operator $S$, rather few properties
are known. The goal of this paper is to prove two identities, {\eqref{eq:center} and \eqref{eq:moment}}, which
appear to be new, and to investigate numerically the behaviour of $S$,
by using the lens transform, recalled below.

\begin{remark}
  The scattering operator can also be defined in some situations for
  the focusing nonlinear Schr\"odinger equation,
  \begin{equation*}
    i\d_t u +\frac{1}{2}\Delta u = -|u|^{2\si}u.
  \end{equation*}
  Statements in such a case require more care, due to the existence
  of solitary waves, and also of
  blow-up solutions when $\si\ge 2/d$.  We 
  choose to consider mostly the defocusing case in order to make the
  presentation as simple as possible. However in the cubic
  one-dimensional case (corresponding to long range scattering), we
  consider numerically the focusing case in
  Section~\ref{sec:long-foc}, and explore whether the size 
  of the (explicit) ground states is a threshold for the existence of a
  (modified) scattering operator. 
\end{remark}
\subsection{Some known dynamical properties}

 It is well-known that the scattering operator $S$, when it is defined
 like above, is unitary on $L^2$,
 and on $\dot H^1$: 
\begin{equation*}
  \lVert S(u_-)\rVert_{L^2(\R^d)}=\lVert u_-\rVert_{L^2(\R^d)}\quad
  ;\quad \lVert \nabla S(u_-)\rVert_{L^2(\R^d)}=\lVert \nabla
  u_-\rVert_{L^2(\R^d)}. 
\end{equation*}
These identities follow from \cite{GV79Scatt,TsutsumiSigma,HT87}, and
can be found also in \cite{CazCourant}.
\smallbreak

A natural question is whether the scattering operator is trivial
(equal to the identity) or not. It is rather straightforward to
compute the first two terms of the asymptotic expansion of $S$ near
the origin, which coincide with the first two Picard
iterates: for $u_-\in \Sigma$ and $\si$ satisfying \eqref{eq:wave-op}, 
\begin{equation*}
  S(\eps u_-) = \eps u_- + \eps^{2\si+1} G(u_-) +
  \O\(\eps^{4\si+1}\)\quad\text{in }L^2(\R^d),\text{ as }\eps\to 0,
\end{equation*}
where
\begin{equation*}
 G(u_-)= -i\int_{-\infty}^{+\infty} U_0(s)\( |U_0(-s)u_-|^{2\si}
 U_0(-s)u_-\)ds,
\end{equation*}
see \cite{CG09} and references therein.
By considering for $u_-$ a Gaussian (whose evolution under $U_0$ can
be computed explicitly), it is easy to check that the operator $G$ is
not identically zero.
\smallbreak

We note also that the properties of $S$ near the origin have proven
useful in inverse problems (see
e.g. \cite{Sasaki2007,ChenMurphy,MR4576319,KiMuVi-p}); the analyticity
of $S$ was addressed in \cite{CG09} (when $\si$ is an integer).
\begin{remark}[Inhomogeneous nonlinearity]\label{rem:inhom1}
  Properties of the scattering map in the case of inhomogeneous
  nonlinearities of the form
  \begin{equation}\label{eq:inhom}
    i\d_t u +\frac{1}{2}\Delta u = a(x)|u|^{2\si}u,
  \end{equation}
  have also proven useful in inverse problems, see e.g. \cite{Murphy2023} 
  and reference therein. In principle, our numerical approach can be
  considered in this case as well, up to a substantial increase of the
  numerical cost. See Remark~\ref{rem:inhom2} for a more precise
  discussion regarding this statement. 
\end{remark}
\smallbreak

As noticed in \cite{CW92} (see also \cite{CazCourant}), if $\conj$
denotes the conjugation $f\mapsto \overline 
f$, the time reversibility property associated
to \eqref{eq:nls} implies the relation
\begin{equation}\label{eq:conj}
  W_\pm = \conj \circ W_\mp\circ\conj .
\end{equation}
In the special $L^2$-critical case $\si=2/d$, two additional
properties have been established: in \cite{COMRL}, we find the relation
  \begin{equation*}
    \F\circ W_\pm^{-1}=W_\mp\circ \F,
  \end{equation*}
which, together with \eqref{eq:conj}, implies
\begin{equation*}
 W_\pm^{-1} = \(\conj  \F\)^{-1}  W_\pm \(\conj \F\). 
\end{equation*}
In \cite{CaDPDE}, still for the case $\si=2/d$, it was proven that for
all $\theta\in [0,2\pi]$, there 
exists infinitely many $u_-\in \Sigma$ such that
$S(u_-)=e^{i\theta}u_-$ (see Proposition~\ref{prop:CaDPDE} below). In
the next section, we present other 
algebraic relations which can be viewed as extra rigidity properties of
the scattering operator $S$ and the waves operators $W_\pm$.

\subsection{New algebraic properties}
\label{sec:new-alg}

The following statement recalls the algebraic properties evoked above,
and contains two new identities:

\begin{theorem}\label{theo:identities}
Let $\frac{1}{d}<\si<\frac{2}{(d-2)_+}$ and $u\in C(\R;\Sigma)$ solution to
\eqref{eq:nls}. Assume that there exist $u_\pm \in \Sigma$ 
such that
\begin{equation*}
  \|U_0(-t)u(t)-u_\pm\|_\Sigma \Tend t {\pm \infty}0. 
\end{equation*}
Then, with $u_0=u_{\mid t=0}$, the following
  identities hold:
   \begin{equation}
    \label{eq:mass}
    \|u_-\|_{L^2}= \|u_0\|_{L^2}=\|u_+\|_{L^2},
  \end{equation}
  \begin{equation}
    \label{eq:energy}
    \|\nabla u_-\|^2_{L^2}= \|\nabla
    u_0\|^2_{L^2}+\frac{2}{\si+1}\|u_0\|_{L^{2+2\si}}^{2+2\si}
    =\|\nabla u_+\|^2_{L^2},
  \end{equation}
  \begin{equation}
    \label{eq:moment0}
    \begin{aligned}
       \IM  \int_{\R^d} \overline{u_-}(x)\nabla u_-(x)dx& =
    \IM\int_{\R^d}\overline{u_0}(x)\nabla u_0(x)dx\\
&=   \IM  \int_{\R^d} \overline{u_+}(x)\nabla u_+(x)dx,
    \end{aligned}
   \end{equation}
  \begin{equation}
    \label{eq:center}
  \int_{\R^d} x|u_-(x)|^2dx =    \int_{\R^d} x|u_0(x)|^2dx =
  \int_{\R^d} x|u_+(x)|^2dx .
\end{equation}
In the $L^2$-critical case $\si=2/d$, we have in addition
  \begin{equation}
    \label{eq:moment}
    \|xu_-\|_{L^2}^2 +
    \frac{2d}{2+d}\|\widehat{u_-}\|_{L^{2+4/d}}^{2+4/d}=
    \|xu_0\|_{L^2}^2 =  \|xu_+\|_{L^2}^2 +
    \frac{2d}{2+d}\|\widehat{u_+}\|_{L^{2+4/d}}^{2+4/d}.
  \end{equation}
\end{theorem}
\begin{remark}
  The above result holds (at least) in the following case. Let $u_-\in
  \Sigma$. For $\si<\frac{2}{(d-2)_+}$ satisfying
  \eqref{eq:completeness}, or more generally \eqref{eq:wave-op} provided that
  $\|u_-\|_\Sigma$ is sufficiently small, we let 
 $u_0= W_-u_-\in \Sigma$, and $u_+=Su_-\in
  \Sigma$.
\end{remark}

The identities \eqref{eq:mass}, \eqref{eq:energy} and
\eqref{eq:moment0} are standard.  The identities \eqref{eq:center} and
\eqref{eq:moment} are new to the best of our knowledge. The relation \eqref{eq:moment} can
be proven by using the 
pseudo-conformal conservation law discovered in \cite{GV79Scatt}. We
give an argument based on the lens transform, which makes the proof
very short.

The relation \eqref{eq:center} implies for instance (together with
\eqref{eq:mass}) the fact that the nonlinear scattering map cannot act
as a translation 
operator on a nontrivial asymptotic state. This is obvious in the case of a
radially symmetric state, since the dynamics of \eqref{eq:nls}
preserves the radial symmetry, and we have more generally: 
\begin{corollary}
  Let $\frac{1}{d}<\si<\frac{2}{(d-2)_+} $ satisfying \eqref{eq:completeness}, and
    $x_0\in \R^d$. The identity 
  \begin{equation*}
    S(u_-)(x)=u_-(x-x_0),\quad \forall x\in \R^d,
  \end{equation*}
  is possible for some $u_-\in \Sigma$ only in the trivial case, $u_-=0$ or
  $x_0=0$. 
\end{corollary}

\begin{remark}
  The Plancherel formula shows that \eqref{eq:moment0} can be rewritten,
  since 
  \begin{equation*}
      \IM  \int_{\R^d} \overline{u_-}(x)\nabla u_-(x)dx=\frac{1}{i}\int_{\R^d}
      \overline{u_-}(x)\nabla u_-(x)dx, 
  \end{equation*}
as
\begin{equation*}
    \int_{\R^d} \xi|\widehat{u_-}(\xi)|^2d\xi =    \int_{\R^d}
    \xi|\widehat{u_0}(\xi)|^2d\xi = 
  \int_{\R^d} \xi|\widehat{u_+}(\xi)|^2d\xi .
\end{equation*}
So the centre in phase space is the same for $u_-$, $u_0$, and
$u_+$. This property constitutes a precious test for numerical
methods. 
\end{remark}

\subsection{Lens transform}
\label{sec:lens}

For $|t|<\pi/2$, introduce $v$, defined by
\begin{equation}
  \label{eq:lens}
  v(t,x)=\frac{1}{\(\cos t\)^{d/2}} u\(\tan t,\frac{x}{\cos t}\) e^{-i
  \frac{|x|^2}{2}\tan t},
\end{equation}
which has the same value as $u$ at time $t=0$. As noticed in
\cite{KavianWeissler,Rybin}  (see also \cite{CaM3AS,TaoLens}), $v$
solves (for $|t|<\pi/2$)
\begin{equation}
  \label{eq:nlsharmo-gen}
  i\d_t v +\frac{1}{2}\Delta v = \frac{\lvert x\rvert^2}{2} v + \(\cos
  t\)^{d\si-2}\lvert
  v\rvert^{2\si}v,\quad v_{\mid t=0} = u_{\mid t=0}=u_0.
\end{equation}
We note that the nonlinearity in  \eqref{eq:nlsharmo-gen} is
autonomous exactly in the $L^2$-critical case $\si=2/d$. 
The compactification of time in \eqref{eq:lens} implies the following
consequence: 
\begin{lemma}[See Lemma~2.4 in \cite{CaDPDE}]\label{lem:waveharmo}
  Let  $u\in C(\R;\Sigma)$ and $v\in
  C\(\left[-\frac{\pi}{2},\frac{\pi}{2}\right];\Sigma\)$ solve
  \eqref{eq:nls}  and 
  \eqref{eq:nlsharmo}, respectively, with $u_{\mid t=0}=v_{\mid
  t=0}=u_0$. Suppose that there exist $u_\pm\in\Sigma$ such that
\begin{equation*}
  \left\| U_0(-t)u(t)-u_\pm\right\|_{\Sigma}\Tend t {\pm \infty}0.
\end{equation*}
Then
\begin{equation}\label{eq:initial_final_identities}
 v\(-\frac{\pi}{2},x\)=e^{id\pi/4}{\mathcal F}\(
 u_-\)(-x) \quad ;\quad v\(\frac{\pi}{2},x\)= e^{-id\pi/4} {\mathcal F}\(
  u_+\)(x) .
\end{equation}
\end{lemma}

The above property was used in \cite{CaDPDE} in order to reduce the
property $S(u_-)=e^{i\theta}u_-$ in the $L^2$-critical case to the
existence of solutions to an elliptic equation. It has been employed also
to prove scattering for low regularity solutions to \eqref{eq:nls} by
considering  \eqref{eq:nlsharmo-gen}, in
e.g. \cite{BurqThomann2024,BTT13,PoRoTh14}. It is also present in the argument
of \cite{BGTV23} for the above mentioned result. The lens transform
was used for other questions than scattering theory, see
e.g. \cite{DMR11,TaoLens}. 

\subsection{Numerical approach}
\label{sec:approach}

A major difficulty when trying to simulate numerically the scattering
map $S$ lies in the dispersive properties of the group $U_0$, since
\begin{equation*}
  U_0(t)f(x)\Eq t \infty  \frac{1}{(it)^{d/2}}\widehat
  f\(\frac{x}{t}\)e^{i\frac{\lvert x\rvert^2}{2t}},
\end{equation*}
see Lemma~\ref{lem:tsutsumi} for a precise statement. In particular,
for large time, the solution leaves any fixed box, and the exiting
part cannot be neglected if one wants to take hold of the asymptotic
state. In \cite[Section~5]{CGM3AS}, a first attempt at computing the
scattering operator was made, consisting in 
computing $U_0(-t)u(t)$ for large values of $t$. The role of the
composition with $U_0(-t)$ was to compensate the dispersive effect, in
order to get the convergence to the scattering state and to avoid
boundary effects. However, since $u(t)$ is computed first, dispersive
effects are present, and corrected afterwards, which cannot rule out
the possibility of interaction with the boundary. The way out chosen in
\cite{CGM3AS} was to stop the computation at some large (but fixed)
time $T$, and choose a computational domain sufficiently large so the
interaction of $u$ with the domain remains weak. 
\smallbreak

In contrast, our approach is based on the lens transform which leads to several advantages: firstly, as pointed out in \cite{TaoLens}, the lens
transform \eqref{eq:lens} compactifies space and time; time is
obviously compactified because $t=\infty$ for $u$ corresponds to
$t=\pi/2$ for $v$. Space is compactified too because the Laplace
operator in \eqref{eq:nls} is replaced by the harmonic oscillator in
\eqref{eq:nlsharmo-gen}, and this operator is confining (it has a
compact resolvent). Therefore, it becomes much easier to avoid the
interaction of $v$ with the boundary of the computational domain,
provided that this domain is chosen sufficiently large.
Secondly, as the
solution to \eqref{eq:nlsharmo-gen} can be computed for $t\in
(-\pi/2,\pi/2)$ at least when the cosine factor is locally integrable,
we may investigate the large time behaviour of the solution to
\eqref{eq:nls} for any $\si>1/d$. Note that the cosine factor ceases
to be locally integrable when $\si=1/d$, which corresponds to the
appearance of long range effects in \eqref{eq:nls}, that is, the
notion of scattering must be modified.
\smallbreak

An idea similar to the one described above was used in \cite{CaSu23}
in order to investigate 
numerically the dynamics of the Hamiltonian $-\Delta-|x|^2$ perturbed
by a logarithmic nonlinearity. The Schr\"odinger equation in the
presence of this repulsive harmonic potential enjoys an exponential
dispersive rate (as opposed to the algebraic rate without potential),
which causes huge difficulties when it comes to numerics. An analogue
of the lens transform was used in \cite{CaSu23} in order to neutralise
this strong dispersion: that transform removes the
potential, and compactifies time (so the dispersion due to the
operator $U_0(t)$ is essentially harmless numerically).

\subsection{Contents}

In Section~\ref{sec:background}, we recall some technical tools and
properties useful in the analysis of nonlinear scattering theory, and
prove Theorem~\ref{theo:identities}. We also recall the argument of
the main result from \cite{CaDPDE}, proving the existence of rotating points
in the $L^2$-critical case, and sketch some specific aspects of the long
range case $\si=1/d$. Section~\ref{sec:num} is dedicated to
numerical experiments, with emphasis on the existence of rotating
points in the $L^2$-critical case and in the $L^2$-supercritical
case, and on the long range case in one dimension.
In Section~\ref{sec:short}, we explore numerically the actual
threshold value for scattering to hold in $\Sigma$ for large data. In
Section~\ref{sec:long-foc}, we consider the focusing, cubic,
one-dimensional Schr\"odinger equation, and compare numerically the
existence of modified scattering with the size of the initial data. 
Concluding remarks and a few directions for future research are provided in Section~\ref{sec:conclusions}. In
Appendices~\ref{app:hermite_transform} \& \ref{app:hermit_identities}, we gather auxiliary identities and algorithms
for Hermite expansions used in Section~\ref{sec:num}, and in
Appendix~\ref{app:continuity_argument_sigma}, 
we show a quantitative continuity argument for the dependence of the
scattering operator upon the power $\si$, in order to support the
discussion on some experiments presented in
Section~\ref{sec:num}.

\subsection*{Acknowledgements}
The authors are grateful to the analonymous referees for
their constructive suggestions,
to Valeria Banica and
Katharina Schratz for many helpful discussions at the origin of this
work, and to Marcus Webb for several helpful discussions related to stable Hermite transforms. Both authors gratefully acknowledge funding from the European Research Council (ERC) under the European Union's Horizon 2020 research and innovation programme (grant agreement No.\ 850941).

\section{Analytic background}
\label{sec:background}

\subsection{Generalities}
\label{sec:generalities}

We recall the three standard quantities conserved by the flow
of 
\eqref{eq:nls}: mass, energy, and linear momentum, given respectively
by
\begin{align*}
  M&=\|u(t)\|_{L^2(\R^d)}^2,\\
E& = \frac{1}{2}\|\nabla u(t)\|_{L^2(\R^d)}^2
   +\frac{1}{\si+1}\|u(t)\|_{L^{2\si+2}(\R^d)}^{2\si+2} ,\\
P &= \IM \int_{\R^d} \bar u(t,x)\nabla u(t,x)dx.
\end{align*}

The following lemma is standard (see
\cite{TsutsumiSigma} or  \cite{Rauch91}):
\begin{lemma}\label{lem:tsutsumi}
  Let $f\in L^2(\R^d)$, and recall that
  $U_0(t)=e^{i\frac{t}{2}\Delta}$. 
  \begin{equation*}
    \left\lVert U_0(t) f - A(t)f\right\rVert_{L^2(\R^d)} \Tend t
    {\pm\infty} 0, \quad \text{where }A(t)f(x) = \frac{1}{(it)^{d/2}}\widehat
  f\(\frac{x}{t}\)e^{i\frac{\lvert x\rvert^2}{2t}},
  \end{equation*}
and the Fourier transform is normalised in \eqref{eq:fourier}. 
\end{lemma}
The lemma is an easy consequence of the factorisation
\begin{equation*}
  U_0(t) = M_t D_t {\mathcal F} M_t,
\end{equation*}
where $M_t$ stands for the multiplication by the function
$e^{i\frac{\lvert  x\rvert^2}{2t}}$, ${\mathcal F}$ is the Fourier
transform defined in  \eqref{eq:fourier}, and $D_t$ is the dilation
operator 
\begin{equation*}
      \(D_t f\)(x) = \frac{1}{(it)^{d/2}}f\(\frac{x}{t}\).
\end{equation*}
These three operators are unitary on $L^2$, and the lemma follows from
the Dominated Convergence Theorem, noticing that $A(t) = M_t D_t {\mathcal F}$.

%As $\nabla$ commutes with $U_0(t)$, we infer:
%\begin{lemma}
 % Let $2\le p<\frac{2d}{(d-2)_+}$. For any $f\in H^1(\R^d)$, 
 % \begin{equation*}
 %  \left\lVert \nabla U_0(t) f - i M_t D_t \(x\mathcal F
 %     f\)\right\rVert_{L^2(\R^d)} \Tend t 
 %   {\pm\infty} 0.
 % \end{equation*} 
%\end{lemma}
%\begin{proof}
%Lemma~\ref{lem:tsutsumi} implies, since $\nabla f\in L^2(\R^d)$,
%\begin{equation*}
%  \left\lVert  U_0(t)  \nabla f -  A(t)\nabla f\right\rVert_{L^2(\R^d)}
%    \Tend t     {\pm\infty} 0 .
%\end{equation*}
%  We compute
 % \begin{align*}
  %[\nabla ,A(t)]f &= \nabla A(t)f-A(t)\nabla f =\frac{1}{t}M_t D_t
  %                  \nabla \hat f   
 % \end{align*}
 % We recall that $U_0(t) \nabla f=\nabla U_0(t)f$, and we note the identity
 % \begin{equation*}
 %   A(t)\nabla f=M_t D_t\mathcal F \(\nabla
 %   f\)= iM_t D_t\(x\mathcal F f\),
 % \end{equation*}
 % hence the lemma. 
%\end{proof}
The sharp decay in time of linear solutions is classical and recalled in
the following lemma:
\begin{lemma}\label{lem:disp-lin}
  There exists $C$ such that for all $f\in \Sigma$, and all $2\le
  p<\frac{2d}{(d-2)_+}$, 
  \begin{equation*}
    \|U_0(t) f\|_{L^p(\R^d)}\le
    \frac{C}{\<t\>^{d\(\frac{1}{2}-\frac{1}{p}\)}}\|f\|_{\Sigma},\quad
      \forall t\in \R.
  \end{equation*}
\end{lemma}
\begin{proof}
  The tools and properties evoked in this proof can be found in
  e.g. \cite{CazCourant} or \cite{GinibreDEA}. 
  Gagliardo-Nirenberg inequality yields
  \begin{equation*}
      \|U_0(t) f\|_{L^p(\R^d)}\lesssim
      \|U_0(t)f\|_{L^2(\R^d)}^{1-\delta(p)}
\|\nabla U_0(t)f\|_{L^2(\R^d)}^{\delta(p)},\quad
\delta(p):=d\(\frac{1}{2}-\frac{1}{p}\),
  \end{equation*}
and since $U_0(t)$ is unitary on $L^2$ and commutes with $\nabla$,
 \begin{equation*}
      \|U_0(t) f\|_{L^p(\R^d)}\lesssim
      \|f\|_{L^2(\R^d)}^{1-\delta(p)}
\|\nabla f\|_{L^2(\R^d)}^{\delta(p)},\quad
\forall t\in \R.
  \end{equation*}
For $|t|>1$, we recall that the operator $J(t)=x+it\nabla$ can be
factorised as 
\begin{equation*}
  J(t) f=  it\, e^{i\frac{|x|^2}{2t}}\nabla \(f\,
  e^{-i\frac{|x|^2}{2t}}\), 
\end{equation*}
so  Gagliardo-Nirenberg inequality also yields
\begin{equation*}
  \|f\|_{L^p(\R^d)}\lesssim \frac{1}{|t|^{\delta(p)}}  \|f\|_{L^2(\R^d)}^{1-\delta(p)}
\|J(t) f\|_{L^2(\R^d)}^{\delta(p)}.
\end{equation*}
Moreover, $J(t) = U_0(t)x U_0(-t)$, so, since $U_0(t)$ is unitary on
$L^2$,
\begin{equation*}
   \|U_0(t)f\|_{L^p(\R^d)}\lesssim \frac{1}{|t|^{\delta(p)}}  \|f\|_{L^2(\R^d)}^{1-\delta(p)}
\|x f\|_{L^2(\R^d)}^{\delta(p)},
\end{equation*}
hence the lemma. 
\end{proof}

\subsection{Proof of  Theorem~\ref{theo:identities}}
\label{sec:direct}

The property
\begin{equation*}
  \|U_0(-t)u(t)-u_\pm\|_\Sigma\Tend t {\pm \infty} 0,
\end{equation*}
together with
Lemma~\ref{lem:tsutsumi}, Sobolev embedding, and
Lemma~\ref{lem:disp-lin}, then easily yield \eqref{eq:mass},
\eqref{eq:energy} and \eqref{eq:moment0}. 

Formally, the centre of mass 
  \begin{equation*}
    \overline x(t)= \int_{\R^d}x|u(t,x)|^2dx
  \end{equation*}
evolves according to 
\begin{align*}
  \frac{d}{dt}   \overline x(t) &= 2\RE \int_{\R^d}x\overline
  u(t,x)\d_t u(t,x) dx=-\IM \int_{\R^d}x\overline
  u(t,x) \Delta u(t,x) dx\\
& = \IM \int_{\R^d}\overline
  u(t,x) \nabla u(t,x) dx=P,
\end{align*}
and so its derivative is constant. This identity can be made rigorous
by resuming the argument given in the proof of \cite[Lemma~6.5.2]{CazCourant}, which
consists, in the present case, in considering as an intermediate quantity
\begin{equation*}
   \overline x_\eps(t)= \int_{\R^d}x e^{-\eps|x|^2}|u(t,x)|^2dx.
\end{equation*}
Let
\begin{equation*}
  z(t) = \RE\int_{\R^d}\overline u(t,x)J(t)u(t,x)dx
  = \int_{\R^d}\overline u(t,x)J(t)u(t,x)dx,
\end{equation*}
where $J(t)=x+it\nabla$ was introduced in the proof of
Lemma~\ref{lem:disp-lin}. The reason why the imaginary part of the
above integral is zero is straightforward, by writing
\begin{equation*}
   \int_{\R^d}\overline u(t,x)J(t)u(t,x)dx=  \int_{\R^d}x \overline
   u(t,x)u(t,x)dx+  it \int_{\R^d}\overline u(t,x)\nabla u(t,x)dx.
\end{equation*}
More precisely, expanding,
\begin{equation*}
  z(t) = \overline x(t)-t P=\overline x(0)=\int_{\R^d}x|u_0(x)|^2dx.
\end{equation*}
On the other hand, the factorisation
$J(t) = U_0(t)x U_0(-t)$ and the assumption $\|U_0(-t)u(t)-
u_\pm\|_\Sigma \to 0$ yield, with $\<f,g\>=\int f\overline g$, and since
$U_0(t)^*=U_0(-t)$, 
\begin{align*}
  z(t)  & =  \int_{\R^d}
          \overline{u(t,x)}U_0(t)xU_0(-t)u(t,x)dx=\<U_0(t)xU_0(-t)
          u(t),u(t)\>\\
  &=\<xU_0(-t)u(t),U_0(-t)u(t)\>
  \Tend t {\pm \infty} \<x u_\pm,u_\pm\> = \int_{\R^d} x |u_\pm(x)|^2dx,
\end{align*}
hence \eqref{eq:center}.
\bigbreak

To prove \eqref{eq:moment}, we use the lens transform \eqref{eq:lens}.
We first emphasise
the fact that the lens transform yields an alternative proof of
\eqref{eq:moment0}-\eqref{eq:center}, since, like in the linear case
(discovered in \cite{Ehrenfest}), $v$, solution to
\eqref{eq:nlsharmo-gen}, satisfies
\begin{align*}
 & \frac{d}{dt}\IM \int_{\R^d} \overline v(t,x)\nabla v(t,x)dx =
  \int_{\R^d} x |v(t,x)|^2dx,\\ 
&\frac{d}{dt} \int_{\R^d} x
  |v(t,x)|^2dx= \IM \int_{\R^d} \overline v(t,x)\nabla v(t,x)dx . 
\end{align*}
This implies
\begin{align*}
 & \IM \int_{\R^d} \overline v(t,x)\nabla v(t,x)dx=\cos t  \times\IM \int_{\R^d}
  \overline u_0(x)\nabla u_0(x)dx +\sin t \times\int_{\R^d} x
   |u_0(x)|^2dx,\\
&\int_{\R^d} x |v(t,x)|^2dx= \cos t \times \int_{\R^d} x
   |u_0(x)|^2dx + \sin t \times\IM \int_{\R^d}
  \overline u_0(x)\nabla u_0(x)dx ,
\end{align*}
so \eqref{eq:moment0}-\eqref{eq:center} follow from Plancherel equality
and Lemma~\ref{lem:waveharmo}, since
\begin{equation*}
   \IM \int_{\R^d} \overline v(t,x)\nabla v(t,x)dx =\frac{1}{i}
   \int_{\R^d} \overline v(t,x)\nabla v(t,x)dx . 
\end{equation*}
\smallbreak

In the case $\si=2/d$, \eqref{eq:nlsharmo-gen} is autonomous, 
\begin{equation}
  \label{eq:nlsharmo}
  i\d_t v +\frac{1}{2}\Delta v = \frac{\lvert x\rvert^2}{2} v + \lvert
  v\rvert^{4/d}v,\quad v_{\mid t=0} = u_{\mid t=0}=u_0,
\end{equation}
and 
we know that if $v_{\mid t=0}\in \Sigma$, then \eqref{eq:nlsharmo} has
a unique, global solution $v\in \C(\R;\Sigma)\cap C^1(\R;\Sigma^*)$ (see
\cite[Section~9.2]{CazCourant} or \cite{CaAHP}).

\begin{lemma}[From Lemma~3.1 in \cite{CaAHP}]
  Let $u_0\in \Sigma$: the solution $v\in C(\R;\Sigma)$ to
  \eqref{eq:nlsharmo} satisfies
  \begin{align}
    &\frac{d}{dt}\( \| xv(t)\sin t-i\cos t \nabla v(t)\|_{L^2}^2
      +\frac{2d}{2+d}\cos^2t\|v(t)\|_{L^{2+4/d}}^{2+4/d}\)
      =0, \label{eq:evol1}\\ 
   &\frac{d}{dt}\( \| xv(t)\cos t+i\sin t\nabla v(t)\|_{L^2}^2
      + \frac{2d}{2+d}\sin^2t\|v(t)\|_{L^{2+4/d}}^{2+4/d}\)
     =0. \label{eq:evol2} 
  \end{align}
\end{lemma}

We can now complete the proof of Theorem~\ref{theo:identities}.
  For $u_-\in \Sigma$, let $u\in C(\R;\Sigma)$ be
  given by the wave operator, with $u_{\mid t=0}=u_0$. We define $v$ as the solution to
  \eqref{eq:nlsharmo}. Considering the constant of motion appearing in
  \eqref{eq:evol1} at times $t=-\pi/2$, $0$ and $\pi/2$, we get
  \begin{equation*}
    \left\|x v\(-\frac{\pi}{2}\)\right\|_{L^2}^2 =
    \|\nabla u_0\|_{L^2}^2 + \frac{2d}{2+d}\|u_0\|_{L^{2+4/d}}^{2+4/d}
   =\left\|x v\( \frac{\pi}{2}\)\right\|_{L^2}^2.
 \end{equation*}
 Using Lemma~\ref{lem:waveharmo}, this yields \eqref{eq:energy}.

 Similarly, considering the constant of motion appearing in
 \eqref{eq:evol2} at times $t=-\pi/2$, $0$ and $\pi/2$,
   \begin{align*}
    \left\|\nabla v\(-\frac{\pi}{2}\)\right\|_{L^2}^2 +
     \frac{2d}{2+d}\left\|v\(-\frac{\pi}{2}\)
     \right\|_{L^{2+4/d}}^{2+4/d} &=\|x u_0\|_{L^2}^2\\
     &=
  \left\|\nabla v\(\frac{\pi}{2}\)\right\|_{L^2}^2 +
    \frac{2d}{2+d}\left\|v\(\frac{\pi}{2}\)\right\|_{L^{2+4/d}}^{2+4/d}
  , 
 \end{align*}
and Lemma~\ref{lem:waveharmo}  yields \eqref{eq:moment}.

\subsection{Rotating points for the scattering operator in the
  $L^2$-critical case}
\label{sec:rotat}

We recall the main result from \cite{CaDPDE}, and sketch the main arguments
of the proof:
\begin{proposition}[From \cite{CaDPDE}]\label{prop:CaDPDE}
  Let $d\ge 1$ and $\si=2/d$ in \eqref{eq:nls}. For any $\theta\in
  [0,2\pi)$, there exist infinitely many functions $u_-\in \Sigma$
  such that $S(u_-)=e^{i\theta}u_-$. 
\end{proposition}
As \eqref{eq:nls} is invariant under the gauge transform $u\mapsto
e^{i\omega}u$ for any constant $\omega\in \R$, we emphasise that the above
statement holds even modulo this equivalence relation. The proof
actually allows to consider $u_0=W_-u_-(=\psi$ below) real-valued. 
\begin{proof}[Sketch of the proof]
  Using the lens transform \eqref{eq:lens}, Lemma~\ref{lem:waveharmo}
  shows  that it suffices to  construct a solution $v$ to
  \eqref{eq:nlsharmo-gen} such that
  \begin{equation*}
    v\(\frac{\pi}{2},x\) = e^{i\theta}
      e^{-id\pi/2}v\(-\frac{\pi}{2},-x\),\quad \forall x\in \R^d.
  \end{equation*}
Since  \eqref{eq:nlsharmo-gen} is autonomous in the $L^2$-critical
case $\si=2/d$, we can seek such a special solution of the form
$v(t,x) = e^{-i\nu t} \psi(x)$, where $\psi$ is a radial solution to
the elliptic problem
\begin{equation}\label{eq:ell-crit}
  \nu \psi = H\psi+|\psi|^{4/d}\psi,\quad
  H=-\frac{1}{2}\Delta+\frac{|x|^2}{2}, 
\end{equation}
and $\nu$ belongs to the set
\begin{equation*}
  \left\{ \frac{d}{2}-\frac{\theta}{\pi}+2j,\quad
    j\in\N\setminus\{0\}\right\}. 
\end{equation*}
It is then easy to see that for any $\nu>d/2$ and any
$0<\si<\frac{2}{(d-2)_+}$, there exists a radial solution $\psi\in
  \Sigma$ to 
  \begin{equation}\label{eq:ell-gen}
    \nu \psi = H\psi+|\psi|^{2\si}\psi,
  \end{equation}
by minimising the action of the linear case
\begin{equation*}
  I(\psi) = \<H\psi,\psi\>-\nu\<\psi,\psi\>,
\end{equation*}
under the nonlinear constraint 
\begin{equation*}
  \psi\in M=\left\{ \psi\in \Sigma,\ \psi(x)=\psi(|x|)\ ; \
    \frac{1}{1+\si}\int_{\R^d} |\psi(x)|^{2+2\si}=1\right\}.
\end{equation*}
The uncertainty principle and the compactness of the resolvent of $H$
imply that the infimum is negative, and achieved by a function $\tilde
\psi\in M$. One concludes by examining the sign of the Lagrange
multiplier and using the homogeneity of the nonlinearity. 
\end{proof}

We emphasise that the elliptic argument does not use the fact that the
nonlinearity is $L^2$-critical, as we solve \eqref{eq:ell-gen}. This
property is crucial however to 
reduce the proof of Proposition~\ref{prop:CaDPDE} to the construction
of a (radial) solution to \eqref{eq:ell-gen}, since if $\si\not =
2/d$, the nonlinearity in the equation obtained after the lens transform,
\eqref{eq:nlsharmo-gen}, depends on time, and it is hopeless to seek a
special solution in the form of a solitary wave $v(t,x) = e^{-i\nu
  t} \psi(x)$. It is therefore unclear whether the statement of
Proposition~\ref{prop:CaDPDE} can be extended to other powers than
$\si=2/d$ or not. One might think for instance of an argument based on 
Leray-Schauder degree theory (see e.g. \cite{Cronin,Mawhin}), to
change the power of $\cos t$ from zero 
(like in the proof of Proposition~\ref{prop:CaDPDE}) to $d\si-2$ like
in \eqref{eq:nlsharmo-gen}. The numerical simulations that we present
in Section~\ref{sec:num} suggest that this strategy may not be
successful, and that Proposition~\ref{prop:CaDPDE} may be bound to the
$L^2$-critical case.

\subsection{Long range scattering}
\label{sec:long-range}

In this section, we discuss only the one-dimensional case $d=1$. As
recalled in the introduction, when $\si=1$ (cubic nonlinearity),
\begin{equation}
  \label{eq:cubic1D}
    i\d_t u +\frac{1}{2}\d_{x}^2u = |u|^2u,\quad x\in \R, 
\end{equation}
the
solution $u$ cannot behave asymptotically like the
linear evolution of a scattering state $u_+$, unless we deal with the
zero solution. The large time dynamics of $u$ becomes more complicated
to describe than with the mere linear group $U_0(t)$.
Without entering into details, we summarise the main results from
e.g. \cite{HN98,HN06,KatoPusateri2011,Ozawa91}.
\smallbreak

Given asymptotic states $u_-$ and $u_+$, the long range phase
corrections (different whether $t\to -\infty$ of $t\to +\infty$) are
given by
\begin{equation}
  \label{eq:mod-phase}
  \Phi_\pm(t,x):=\mp \left| \widehat{u}_\pm\left( \frac{x}{t}
  \right)\right|^2 \log |t|.
\end{equation}
As proven initially in \cite{Ozawa91},  given $u_-$ (sufficiently
smooth, localised, and small), there exists a unique solution $u$ to
\eqref{eq:nls} such that   
\begin{equation*}
  u(t,x)\Eq t {-\infty} e^{i\Phi_-(t,x)}U_0(t)u_-(x)\Eq t {-\infty}
  e^{i\Phi_-(t,x)+i\frac{x^2}{2t}} \frac{1}{(it)^{1/2}}\widehat 
  u_-\(\frac{x}{t}\).
\end{equation*}
Note
that the phase modification $\Phi_-$ is by no means negligible 
as $t\to -\infty$: it encodes long range effects. The above mentioned
smallness assumption was removed very recently, in \cite{KM-p}, under
various (limited) smoothness assumptions. 
\smallbreak

The modified asymptotic completeness is similar. As proven initially
in \cite{HN98}, given $u_0$
(sufficiently smooth, localised, and small), there exists $u_+$ such
that the solution $u$ to 
\eqref{eq:cubic1D} with $u_{\mid t=0}=u_0$ satisfies
\begin{equation*}
  u(t)\Eq t {+\infty} e^{i\Phi_+(t)}U_0(t)u_+.
\end{equation*}
We note that the argument in the proof of \cite[Theorem~1.2]{HN98}
(see also \cite{KatoPusateri2011}) goes as follows:
\begin{itemize}
\item First, it is proven that, setting $f(t)  = U_0(-t)u(t)$, there
  exists $W\in L^\infty\cap L^2$ and $\beta>0$ such that
  \begin{equation*}
    \widehat f(t)\exp\( -i\int_1^t |\widehat u(s)|^2\frac{ds}{s}\)-
    W=\O\(t^{-\beta}\)\quad\text{as }t\to \infty, \text{ in }L^2\cap L^\infty.
  \end{equation*}
\item Then, it is established that there exists $\varphi\in L^\infty$
  such that
  \begin{equation*}
    \int_1^t |\widehat u(s)|^2\frac{ds}{s} - |W|^2\log t 
     -\varphi=\O\(t^{-\beta}\),\quad\text{in } L^\infty.
   \end{equation*}
\item In view of these two results,
one infers (using a refined version of Lemma~\ref{lem:tsutsumi}),
\begin{equation*}
  u(t,x)=\frac{1}{(it)^{1/2}}W\(\frac{x}{t}\) \exp\(
  i\frac{x^2}{2t} -i \left|W\(\frac{x}{t}\)\right|^2\log t
  -i\varphi\(\frac{x}{t}\)\)+\rho(t,x), 
\end{equation*}
with
\begin{equation*}
  \|\rho(t)\|_{L^2}=\O\(t^{-\beta}\),\quad
   \|\rho(t)\|_{L^\infty}=\O\(t^{-1/2-\beta}\).
\end{equation*}
\end{itemize}
In particular, Lemma~\ref{lem:tsutsumi} shows that the modified
asymptotic state $u_+$ is given by 
$\widehat u_+ = W e^{-i\varphi}$. We emphasise the fact that in the
argument from \cite{HN98} (or \cite{KatoPusateri2011}), the rate $\beta$ may depend on the size of
$u_0$.
%\begin{remark}
%  Proceeding like for the proof of \eqref{eq:center}, the identity
%  \begin{equation*}
%    \frac{d}{dt}\overline x(t) = P
%  \end{equation*}
%  remains true (scattering is not used at this step), so $z(t)\equiv
%  \overline x(0)$ still holds. On the other hand, formally,
%  \begin{equation*}
%    z(t)= \<xU_0(-t)u(t),U_0(-t)u(t)\>\Eq t {\pm\infty} \int_\R x\left|
%      \F^{-1} \(\widehat u_\pm e^{i\Phi_\pm(t)}\)(x)\right|^2dx,
%  \end{equation*}
%  and the last expression does not seem obvious to simplify, due to
%  the long range phase modification. 
%\end{remark}
\smallbreak

When using the lens transform, a feature of the cubic one-dimensional
case is that the cosine factor in the nonlinearity in
\eqref{eq:nlsharmo-gen} ceases to be integrable near $t=\pm
\pi/2$. Indeed, as the result by Barab \cite{Barab} evoked in the
introduction shows that one cannot have $u(t)\sim U_0(t)u_-$ as $t\to
-\infty$, unless $u=u_-\equiv 0$, Lemma~\ref{lem:waveharmo}
suggests that $v$ may indeed have some singularity as $t\to \pm
\pi/2$. This is consistent with the fact that a nontrivial phase shift
$\Phi_\pm$ must be introduced in order to define the modified
scattering theory in this case. 
\smallbreak

In view of the construction of $u_+$ described above, and of
Lemma~\ref{lem:waveharmo}, when dealing
with \eqref{eq:nlsharmo-gen} in the case $d=\si=1$, it is rather
natural to first consider 
\begin{equation*}
  w(t,x) = v(t,x) \exp \( i\int_0^t |v(s,x)|^2 \frac{ds}{\cos s}\).
\end{equation*}
Indeed, in view of \eqref{eq:lens},
Lemma~\ref{lem:tsutsumi},  and the construction of the modified
scattering operator, we have
\begin{align*}
 \int_0^t |v(s,x)|^2 \frac{ds}{\cos s}&\Eq t {\frac{\pi}{2}} \int_0^t
 |\widehat u_+(x)|^2 \frac{ds}{\cos s}\Eq t {\frac{\pi}{2}} |\widehat
                                       u_+(x)|^2\log\(\frac{\pi}{2}-t\)\\
&\Eq t {\frac{\pi}{2}} |\widehat u_+(x)|^2\log \tan t,
\end{align*}
as well as 
\begin{equation}\label{eq:v-long}
  v(t,x)\Eq t {\frac{\pi}{2}}\widehat u_+(x) e^{ -i |\widehat
  u_+(x)|^2 \log \tan t},
\end{equation}
which confirms the fact that $v$ is singular as $t\to \pi/2$, but
$w$ is not,
\begin{equation*}
  w(t,x)\Tend t {\frac{\pi}{2}}\widehat u_+(x).
\end{equation*}
We also check that $w$ solves
\begin{equation*}
  i\d_t w = \( Hv\) \exp \( i\int_0^t |v(s,x)|^2 \frac{ds}{\cos
    s}\)= \( Hv\) \exp \( i\int_0^t |w(s,x)|^2 \frac{ds}{\cos
    s}\),
\end{equation*}
where $H=-\frac{1}{2}\d_x^2+\frac{x^2}{2}$. As $w$ is continuous at
$t=\pm \pi/2$, this formulation could be expected to be more tractable
numerically, and compatible with iterative process. However, due to the interaction of $w$ with $v$ in the above formulation, our current numerical approach does not directly benefit from this observation. Instead, in order to
get stable numerical simulations, we have kept $v$ as the main object
of study. 

The case $d=\si=1$ turns out to enjoy a special feature, in the sense
that the equation is known to be completely integrable, after
\cite{ZS}. We note that in the critical case for long range scattering in
higher dimension, $\si=1/d$ for $d\ge 2$, no such property is
known. Many works have investigated the properties of the
one-dimensional cubic Schr\"odinger flow by using complete
integrability. Using specific techniques related to  inverse
scattering and Riemann–Hilbert problems, solutions are computed
numerically in \cite{TrogdonOlver}. However, the analysis of
scattering in the sense considered here is not addressed there, and
comparing the two numerical approaches as well as the corresponding
experiments is not straightforward. This is therefore a natural and challenging
question for future researches.

\section{Numerical computation of the scattering operator}
\label{sec:num}
\subsection{Methodology}\label{sec:methodology} As mentioned in Section~\ref{sec:approach} we can exploit the lens transform \eqref{eq:lens} to compactify time and localise solution values in space, thus providing the basis for efficient numerical computation of the scattering operator $S$ (cf. \eqref{eq:scattering_operator}). For this we can use the following procedure:
\begin{enumerate}[(I)]
	\item For a given initial condition $u_-$ we compute the initial state $v(-\pi/2,x)$ using \eqref{eq:initial_final_identities};\label{list:ouralgo1}
	\item We propagate this initial state forward to $t_{\mathrm{lens}}=\pi/2$ using \eqref{eq:nlsharmo-gen};\label{list:ouralgo2}
	\item Finally, we compute $S(u_-):=u_+$ using \eqref{eq:initial_final_identities}.\label{list:ouralgo3}
\end{enumerate}
A natural spatial discretisation for \eqref{eq:nlsharmo-gen} is a Hermite spectral method (cf. \cite{ThCaNe2009}) with the following $L^2(\R^d)$-orthonormal basis:
\begin{align*}
\mathcal{H}_m(x)=\prod_{j=1}^d\left(\mathrm{H}_{m_j}(x_j) \mathrm{e}^{-\frac{1}{2}x_j^2}\right), \quad m\in\N^{d},
\end{align*}
where $\mathrm{H}_{m_j}$ denotes the Hermite polynomial of degree $m_j\in\N$ normalised with respect to the weight $w(x)=\exp(-x^2)$. The $\mathcal{H}_m$ are eigenfunctions of the Schr\"odinger operator with Harmonic potential, $H$, with
\begin{align}\label{eq:eigenfunction_property}
	H (\mathcal{H}_m)=\left(\frac{d}{2}+\sum_{j=1}^d m_j\right) \mathcal{H}_m,
\end{align}
thus forming a convenient basis for a splitting approach to \ref{list:ouralgo2}. Moreover, this basis allows for the direct computation of $\|\cdot\|_{\Sigma}$: introducing for $k\in \N$,
\begin{equation*}
  \Sigma^k = \left\{ f\in L^2(\R^d),\quad
    \|f\|_{\Sigma^k}:=\|f\|_{H^k(\R^d)}+ \||x|^k f\|_{L^2(\R^d)}<\infty\right\},
\end{equation*}
we have the following equivalence of norms (see e.g. \cite{BCM08,Helffer1984})
\begin{equation}\label{eqn:details_sigma_norm_in_hermite_basis}
\|f\|_{\Sigma^k}^2\sim \sum_{m\in \N^d} \lambda_m^k |\alpha_m|^2,\quad
\lambda_m = \frac{d}{2}+\sum_{j=1}^d m_j,\quad
\text{for } f(x) = \sum_{m\in \N^d} \alpha_m  \mathcal{H}_m(x).
\end{equation}
We thus expand our initial conditions in this basis
\begin{align}\label{eq:expansion_of_u-}
	u_-(x)\approx\sum_{\substack{m\in\N^d\\0\le m\le
  M-1}} \alpha_m^-\mathcal{H}_m(x), 
\end{align}
{where $M\geq 1$ is the spectral truncation parameter,} and we note that the basis $\mathcal{H}_m$ satisfies a range of useful
identities which facilitate the quick computation of Fourier
transforms and derivatives of functions of the form
\eqref{eq:expansion_of_u-}. In particular, there is a numerically
stable way of transforming between Hermite coefficients and function
values introduced by \cite{bunck09} and explained in further detail in
\cite{maierhoferwebb25} (cf. also the implementation in Julia in
\cite{QuantumTimeSteppers,FastGaussQuadrature,FastTransforms}) which
is briefly recalled for completeness in
Appendix~\ref{app:hermite_transform}. Some further identities relevant
to the computations conducted in this work are listed in
Appendix~\ref{app:hermit_identities}. As a result of these identities
and of \eqref{eq:eigenfunction_property}, this basis lends itself
naturally to the integration of \eqref{eq:nlsharmo-gen}. In order to
perform step \ref{list:ouralgo2} above, we will thus resort to
Lie-splitting of \eqref{eq:nlsharmo-gen}, by considering the following
two subproblems 
\begin{align}\label{eq:splitting}
	i\partial_t v=Hv,\quad i\partial_t v=(\cos t)^{d\sigma-2}|v|^{2\sigma}v.
\end{align}
Note that there is a second natural choice for the splitting in
\eqref{eq:splitting}, which is to combine the nonlinearity and the
harmonic potential in one subproblem and keep the Laplacian part of
the linear term separate. This second approach lends itself naturally
to a Fourier pseudospectral spatial discretisation, whereby one has to
choose an appropriate bounding box for the domain $\R^d$. In
principle, it was observed by \cite{ThCaNe2009} that the numerical
performances of both methods are comparable, but we note that in our
case it is crucial to localise dispersive effects for the reliable
numerical computation of solutions as $t\rightarrow \pm\infty$ and
thus it is natural to use \eqref{eq:splitting}. 
\begin{remark}\label{rem:inhom2}
  In the case of an inhomogeneous nonlinearity \eqref{eq:inhom}, the
  second equation in \eqref{eq:splitting} becomes, after lens
  transform,
  \begin{equation}\label{eqn:nonlinear_part_inhomogeneous_case}
    i\partial_t v=(\cos t)^{d\sigma-2}a\( \frac{x}{\cos t}\) |v|^{2\sigma}v.
  \end{equation}
  Like in the classical case of splitting methods for \eqref{eq:nls},
  we note that the solution to the above ODE satisfies $\d_t|v|^2 =
  0$ provided that $a$ is real-valued, and so the equation is actually
  a \emph{linear} ODE and it is solved
  by time integration, with the exact solution of \eqref{eqn:nonlinear_part_inhomogeneous_case} being given by
  \begin{align*}
  	v(t,x)=e^{-i\int_0^t(\cos s)^{d\sigma-2}a\( \frac{x}{\cos s}\) ds |v(0,x)|}v(0,x), \quad x\in\mathbb{R}.
  \end{align*}
  In the homogeneous case the integral simplifies to $\int_0^t(\cos s)^{d\sigma-2}ds$ and requires only \textit{one} evaluation per time step. For the inhomogeneous case, the above method can still be applied, but requires $M$ evaluations of such integrals where $M$ is the spectral truncation parameter, thus leading to a substantial increase in cost.
\end{remark}
\subsubsection{Error analysis}\label{sec:error_analysis}

Using the identities in Appendix~\ref{app:hermit_identities}, the
Fourier transforms in steps \ref{list:ouralgo1} \& \ref{list:ouralgo3}
can be performed exactly, so the only error incurred in our method for
computing $S$ is through the spatial discretisation and the splitting
method. For the splitting part, corresponding convergence results are
already available in the literature. In particular, for
high-regularity solutions the convergence of the aforementioned
splitting method was proved by Gauckler \cite{Gauckler2011}. In later
work,  discrete Strichartz estimates were used
in \cite{Ca-p} to prove
convergence of the Lie splitting method in the low-regularity
r\'egime. For completeness we recall the main theorem from
\cite{Gauckler2011}, adapted to our current setting. Note that
\cite{Gauckler2011} is focused on the Strang splitting, but the
corresponding result for the Lie splitting employed in the present
work follows straightforwardly from a similar analysis. In the
following we denote the splitting method approximation of $v_+$ by
$v_+^{M,\tau}$, where we used $M\in \N$ Hermite modes, and where $\tau$
is the timestep. 

\begin{theorem}[From Theorem 3.4 in \cite{Gauckler2011}]\label{thm:Lie_splitting_convergence}
	Suppose for a given $v(-\pi/2,x)$ the corresponding exact solution of \eqref{eq:nlsharmo-gen} is given by $v(\pi/2,x)$ and that $s>\left\lceil\frac{d+1}{2}\right\rceil+2+\frac{2 d}{3}$ is an integer. Let $R_{s+2}=\sup _{-\pi/2 \le t \le \pi/2}\|v(t,\cdot)\|_{\Sigma^{s+2}}$. Then there exist $h_0,M_0>0$ such that the following error bound holds for step sizes $0< h \le h_0$ and number of Hermite modes $M \ge M_0$ :
	$$
		\|v(\pi/2,\,\cdot\,)-v_+^{M,\tau}\|_{L^2}\le C\left(M^{1+\frac{d}{3}-\frac{1}{2} s}+\tau\right),
	$$
	where $C$ only depends on $d, s,$ and $R_{s+2}$.
\end{theorem}
\begin{remark}
	In view of \cite{Ca-p}, the restriction on $s$ in the
        above results can be significantly weakened. By using a
        spectral cut-off, which amounts to imposing a CFL condition in
        the above method (relating $M$ with $\tau$ -- typically $M=1/\tau$
        when $d=1$), it is shown in \cite{Ca-p} that
        for $v(-\pi/2,\cdot)\in \Sigma^2$, 
        the error is again $\O(\tau)$ at time $\pi/2$. 
\end{remark}
Let us now denote the numerical solution of the scattering operator
following the above process by $u_{+}^{M,\tau}$. Then we have the
following immediate corollary: 
\begin{corollary} Suppose $s$ is as in
  Theorem~\ref{thm:Lie_splitting_convergence}, then there are
  $\tilde{M}_0,\tilde{h}_0,\tilde{C}>0$ depending only on $d,s,$ and
  $R_{s+2}$ such that  for step sizes $0< h \le h_0$ and number of
  Hermite modes $M \ge M_0$,
	\begin{align*}
	\|u_+(\,\cdot\,)-u_+^{M,\tau}\|_{L^2}\le C\left(M^{1+\frac{d}{3}-\frac{1}{2} s}+\tau\right).
\end{align*}
\end{corollary}
%\begin{figure}
%	\centering
%	\begin{subfigure}{0.495\textwidth}
%		\centering
%		\includegraphics[width=0.97\textwidth]{images/convergence_plot_L2N_200_tau_lens_ref_0-0001_d_1_sigma_2.pdf}
%		\caption{Error in $L^2$ for general .}
%	\end{subfigure}
%	\begin{subfigure}{0.495\textwidth}
%		\centering
%		\includegraphics[width=0.97\textwidth]{images/convergence_plot_H1N_200_tau_lens_ref_0-0001_d_1_sigma_2.pdf}
%		\caption{Error in $H^1$.}
%	\end{subfigure}
%	\caption{Convergence plots of the scattering operator for initial condition as presented in Fig.~\ref{fig:scattering_operator_visualisation} \georg{[This is a placeholder figure to be updated!]}.}
%\end{figure}

\subsection{Visualisation of the scattering operator}\label{sec:visualisation_of_scattering_operator}
Having described our methodology above, we can now provide a first
visualisation of the scattering operator for a given initial
condition. In Figure~\ref{fig:scattering_operator_visualisation} we
can see an example of the action of the scattering operator on the
function $u_-(x)=e^{ix-(x-1)^2/2}+e^{-(x+2)^2/4}$. 
\begin{figure}[h!]
	\centering
	\begin{subfigure}{0.495\textwidth}
		\centering
		\includegraphics[width=0.97\textwidth]{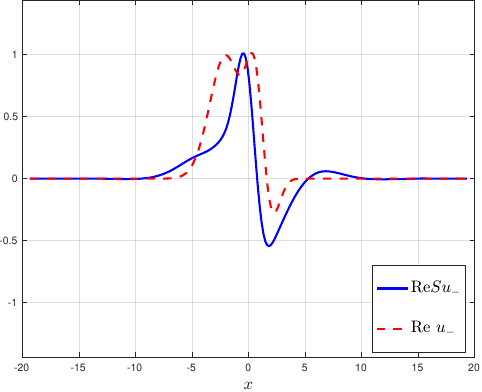}
		\caption{Real part}
	\end{subfigure}
	\begin{subfigure}{0.495\textwidth}
		\centering
		\includegraphics[width=0.97\textwidth]{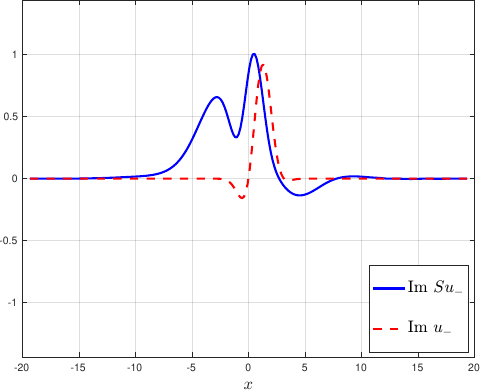}
		\caption{Imaginary part.}
	\end{subfigure}
	\caption{The action of the scattering operator on different initial conditions.}
	\label{fig:scattering_operator_visualisation}
\end{figure}
\subsubsection*{Error in conservation laws} We note that the
conservation laws \eqref{eq:mass}-\eqref{eq:center} provide a useful
way of validating our numerical method. Indeed, we can perform the
following test to see if our numerical scheme (as described in
Section~\ref{sec:methodology}) preserves these quantities: We begin by
initialising $u_-$ in the following form (for us $M=200,d=1$) 
\begin{align*}
	u_-(x)=\sum_{m=0}^{M-1}U_m \mathcal{H}_m(x),
\end{align*}
where $U_m\sim U([0,1]+i[0,1])$ are drawn uniformly at random in the
complex unit square. For each such sample $u_-$ we compute
$u_+:=S(u_-)$ and compare the value of each conserved quantity in
$u_+$ with that of $u_-$, where we use the following notation: 
\begin{align*}
	\mathcal{I}_{1}(u) &=	\|u\|_{L^2},\,\,\,
	\mathcal{I}_2(u) =	\|\nabla u\|^2_{L^2},\,\,\,
	\mathcal{I}_3(u) = \IM  \int_{\R^d} \overline{u}(x)\nabla u(x)dx,\\
	\mathcal{I}_4(u) &=	\int_{\R^d} x|u(x)|^2dx.
\end{align*}
The results are shown in Figure~\ref{fig:conservation_laws} for two
different values of the timestep $\tau_{\mathrm{lens}}$ taken in step \ref{list:ouralgo2} of
our algorithm (cf. Section~\ref{sec:methodology}). The figure shows
that these quantities are indeed preserved to reasonable accuracy,
thus confirming that our methodology for approximating $S$ is
consistent with theoretical properties of the scattering operator. 
\begin{figure}[h!]
	\centering
	\begin{subfigure}{0.495\textwidth}
		\centering
		\includegraphics[width=0.97\textwidth]{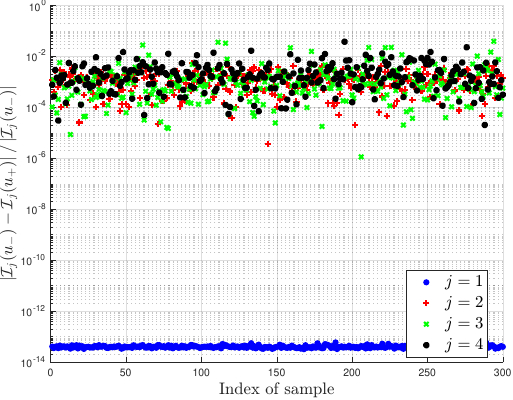}
		\caption{$\tau_{\mathrm{lens}}=0.1, N=100$.}
	\end{subfigure}
	\begin{subfigure}{0.495\textwidth}
		\centering
		\includegraphics[width=0.97\textwidth]{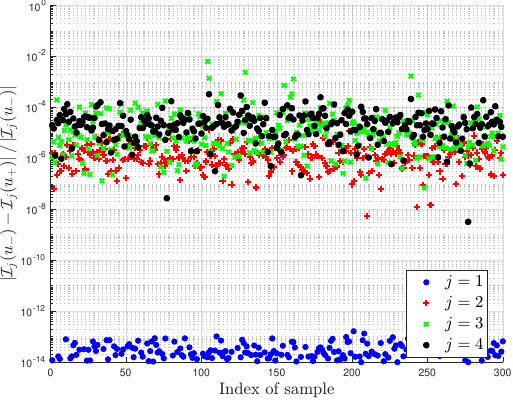}
		\caption{$\tau_{\mathrm{lens}}=0.01, N=200$.}
	\end{subfigure}
	\caption{The error in conservation laws \eqref{eq:mass}-\eqref{eq:center} for a range of randomly chosen initial conditions.}
	\label{fig:conservation_laws}
\end{figure}
%\newpage
\subsection{Rotating-points in the $L^2$-critical case}

As discussed in Section~\ref{sec:rotat},  the scattering operator
$S$ has a selection of fixed/rotating points which are given
as solutions to \eqref{eq:ell-crit}. We can solve this equation
numerically by expanding the unknown rotating point $\psi$ in the
Hermite basis $\{\mathcal{H}_{m}\}_{m\in\N^d, 0\le m\le M-1}$ in the
form 
\begin{align*}
	\psi\approx \sum_{\substack{m\in\N^d\\0\le m\le M-1}}\beta_m \mathcal{H}_m,
\end{align*}
which transforms \eqref{eq:ell-crit} into a finite dimensional, nonlinear system for the coefficients $\beta$. This system can be solved directly using Matlab's \verb|fsolve| command which uses a trust-region algorithm to solve the nonlinear system. With appropriate initialisation (taking $\psi_0$ non-negative), we observed that this solver tends to converge within a few dozen iterations. In practice we found that for higher energy states $j>2$ it is helpful to initialise this nonlinear solver with corresponding lower energy states, i.e. $\psi_0^{(j)}=\psi^{(j-1)}$. The shape of two rotating points and the corresponding action of the scattering operator can be observed in Figure~\ref{fig:rotating_points}.

\begin{figure}[h!]
	\centering
	\begin{subfigure}{0.495\textwidth}
		\centering
		\includegraphics[width=0.97\textwidth]{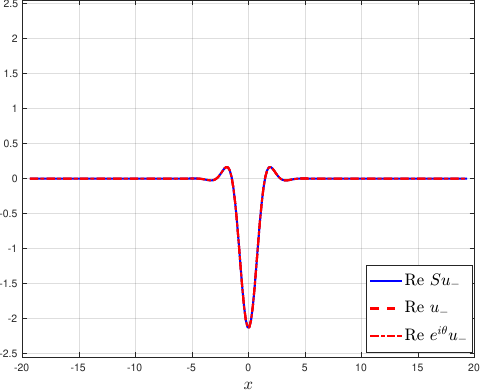}
		\caption{Real part, $j=1,\theta=0$.}
	\end{subfigure}
	\begin{subfigure}{0.495\textwidth}
		\centering
		\includegraphics[width=0.97\textwidth]{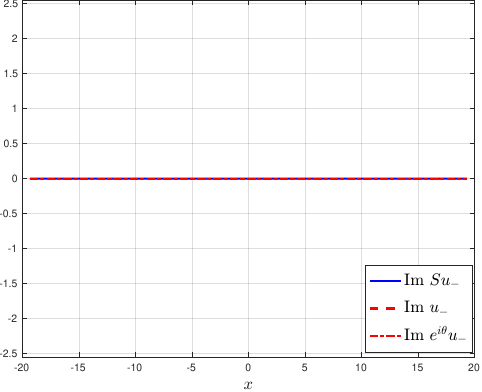}
		\caption{Imaginary part, $j=1,\theta=0.0$.}
	\end{subfigure}
	\begin{subfigure}{0.495\textwidth}
		\centering
		\includegraphics[width=0.97\textwidth]{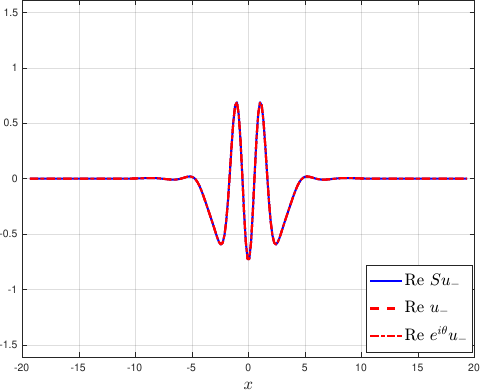}
		\caption{Real part, $j=3,\theta=2.0$.}
	\end{subfigure}
	\begin{subfigure}{0.495\textwidth}
		\centering
		\includegraphics[width=0.97\textwidth]{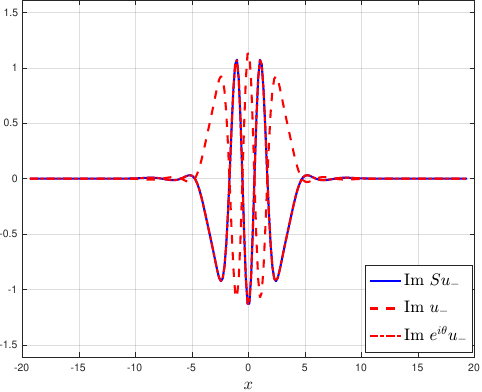}
		\caption{Imaginary part, $j=3,\theta=2.0$.}
	\end{subfigure}
	\caption{The action of the scattering operator on rotating point initial conditions.}
	\label{fig:rotating_points}
\end{figure}

Note since rotating points provide a known reference value for $Su$,
we can use this information to verify the convergence properties of
our methodology. In Figure~\ref{fig:convergence_plots_fixed_points} we
observe the error committed in our approximation of $S$ for both of
the aforementioned fixed points as a function of the timestep in the
lens-transformed system \eqref{eq:nlsharmo-gen} (cf. step (2) in
Section~\ref{sec:methodology}). 

\begin{figure}
	\begin{subfigure}{0.495\textwidth}
		\centering
		\includegraphics[width=0.97\textwidth]{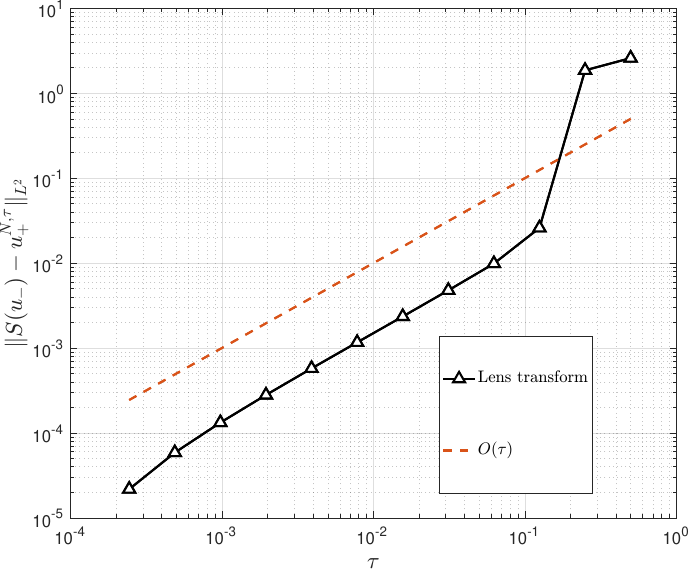}
		\caption{$j=1,\theta=0.0$.}
	\end{subfigure}
	\begin{subfigure}{0.495\textwidth}
	\centering
	\includegraphics[width=0.97\textwidth]{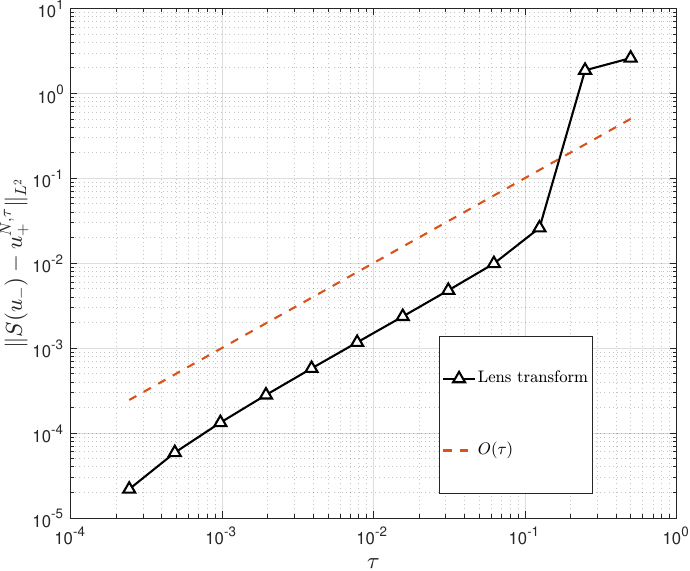}
	\caption{$j=3,\theta=2.0$.}
	\end{subfigure}
	\caption{Convergence behaviour of our method for known reference values in the rotating point case.}
	\label{fig:convergence_plots_fixed_points}
\end{figure}

\subsection{Finding rotating points in supercritical cases}
Having gained confidence in the accuracy and performance of our
numerical approximation of $S$, we can use this technique to
numerically test whether we can find rotating points in the
supercritical case $\sigma>2/d$. To test this we performed the
following steps for a given $\sigma>2/d$. 
\begin{enumerate}
	\item Based on the continuity arguments outlined in
          Appendix~\ref{app:continuity_argument_sigma} we {solve
          \eqref{eq:ell-gen} for} the specific value of $\sigma>2/d$,
          and use this to initialise $u_-$; 
	\item For this value of $u_-$ we try to find the optimal $\theta$ such that \begin{align*}
		\theta_0=\mathrm{argmin}_{\theta\in[0,2\pi)}\left| u_+^{N,\tau_{\mathrm{lens}}}-e^{i\theta}u_-\right|.
	\end{align*}
	\item We optimise both $u_-$ and $\theta$ simultaneously to match the rotating point condition from the initialisation $u_-$ using Matlab's \verb|fsolve|: Let us denote the Hermite coefficients of $u_-$ by $\alpha_m$ such that $u_-=\sum_{m=0}^{M-1}\alpha_m\mathcal{H}_{m}(x)$, and find
	\begin{align*}
		\tilde{\alpha},\tilde{\theta}=\mathrm{argmin}_{\alpha,\theta}\left| u_+^{N,\tau_{\mathrm{lens}}}-e^{i\theta}u_-\right|.
	\end{align*}
\end{enumerate}
The outcomes of all of those steps are shown in
Figure~\ref{fig:rotating_points_supercritical_regime}, and we observe
that while in the critical case (blue solid line with $\sigma=2.0,
\theta=0.5$) our scattering operator recovers $u_-$ to within an error
of size roughly $10^{-6}$, none of our approaches lead to a rotating
point for the supercritical r\'egime (noting that all errors are on the
order of $10^{-2}$, and this error could not be improved in our
experiments, even with finer temporal or spatial resolution in the
solver). While we recognise this is not definite conclusive evidence
(see Appendix~\ref{app:continuity_argument_sigma} for more precise
comments), this experiment provides a numerical indication that if
rotating points exist in the supercritical r\'egime, they may not be
located in the proximity of rotating points from the critical r\'egime. 
\begin{figure}
	\centering
	\includegraphics[width=0.7\textwidth]{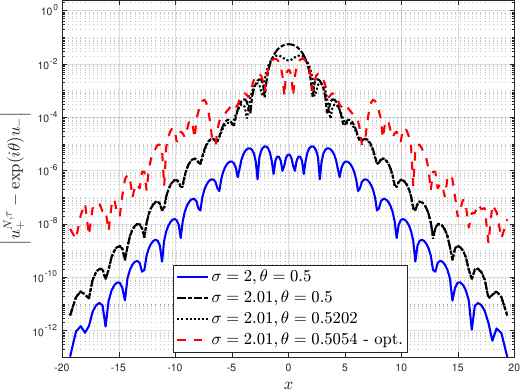}
	\caption{The error observed in pseudo-rotating points in the
          supercritical r\'egime.} 
	\label{fig:rotating_points_supercritical_regime}
\end{figure}

\subsection{Long-range case}

Finally, we note that we can use our approach also to simulate the long-range case. In this case, we can proceed similarly to the algorithm described in Section~\ref{sec:methodology} with the small exception that the Lie splitting of \ref{list:ouralgo2} involves the integration of 
\begin{align*}
	i\partial_t v=(\cos t)^{-1}|v|^{2} v,
\end{align*}
which has a singularity at $t=\pm\pi/2$.
\begin{figure}[h!]
\begin{subfigure}{0.495\textwidth}
	\centering
	\includegraphics[width=0.99\textwidth]{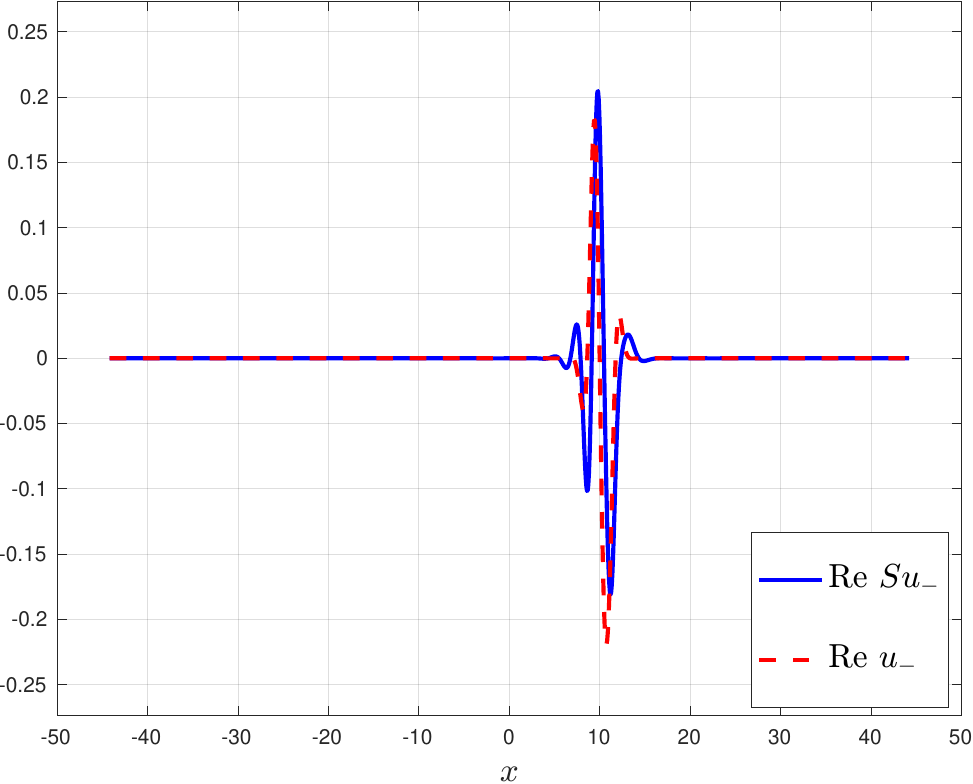}
	\caption{Real part.}
	\label{fig:long_range_real_part}
\end{subfigure}
\begin{subfigure}{0.495\textwidth}
	\centering
	\includegraphics[width=0.99\textwidth]{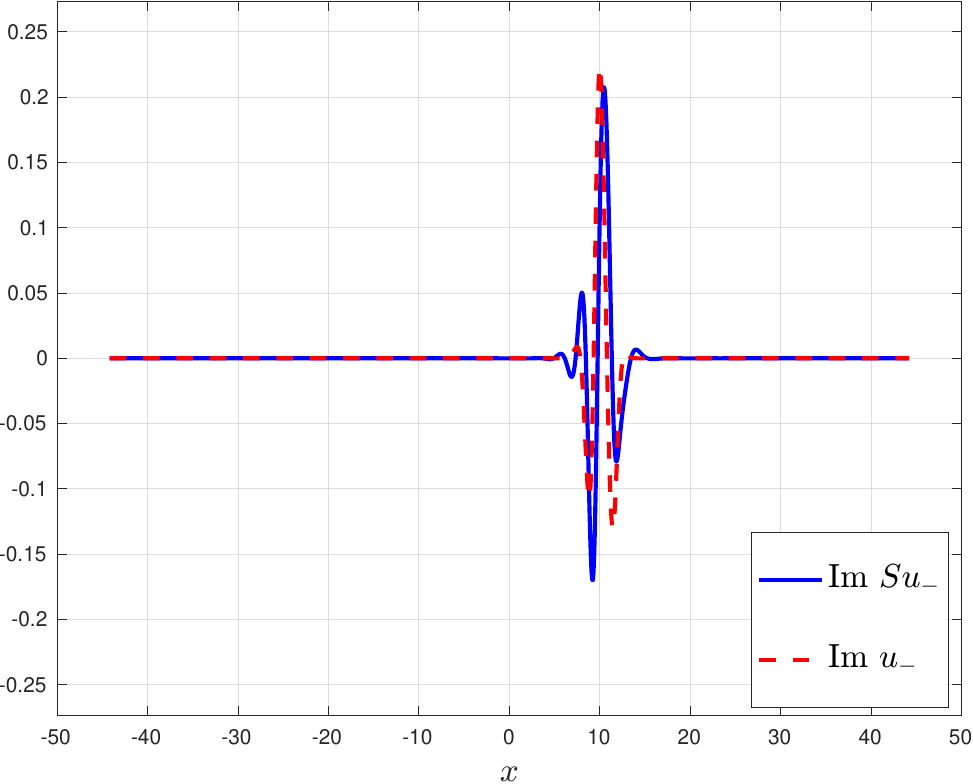}
	\caption{Imaginary part.}
	\label{fig:long_range_imag_part}
\end{subfigure}
\begin{subfigure}{0.495\textwidth}
	\centering
	\includegraphics[width=0.99\textwidth]{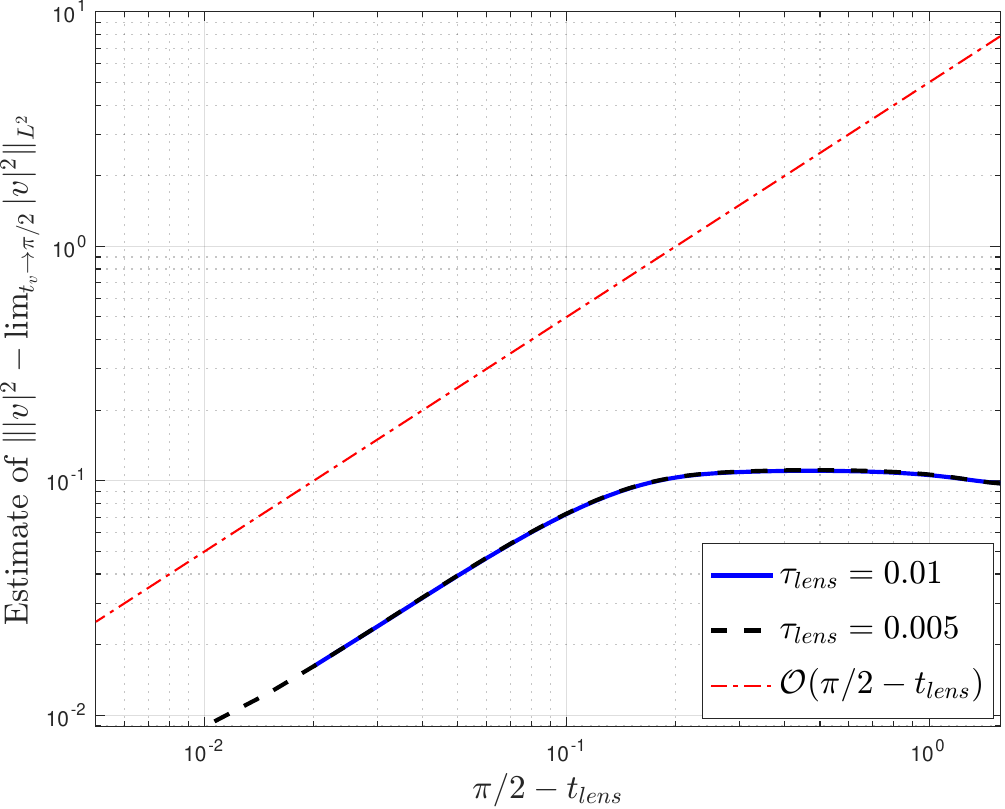}
	\caption{The convergence behaviour of our approach as $t_{\mathrm{lens}}\rightarrow \pi/2$.}
	\label{fig:convergence_long_range}
\end{subfigure}
\caption{The action of the scattering operator and convergence of the numerical method in long-range case.}
\label{fig:long_range_case}
\end{figure}

%\newpage
We overcome this by starting our time-evolution at
$-\pi/2+\tau_{\mathrm{lens}}$ and ending it at
$\pi/2-\tau_{\mathrm{lens}}$, i.e. we omit the first and the final
timestep, and in addition we incorporate the phase shift
\eqref{eq:mod-phase} in steps \ref{list:ouralgo1} \&
\ref{list:ouralgo3}. Figure~\ref{fig:convergence_long_range} provides
numerical evidence that our approach seems to recover the convergence
of the lens-transformed variable in modulus, and thus indicates that
this approach leads to accurate computation of the scattering operator
also in the long-range case. An example of the action of the
scattering operator in this r\'egime can be seen in
figures~\ref{fig:long_range_real_part} \&
\ref{fig:long_range_imag_part}. The numerical study of the long range
case is mainly motivated by the question addressed in
Section~\ref{sec:long-foc}.

\section{Approaching the long range threshold}
\label{sec:short}

As recalled in the introduction, a complete (short range) scattering
theory is available  in the space $\Sigma$ provided that the power
$\si$ satisfies \eqref{eq:completeness}, that is
\begin{equation*}
   \si \ge \frac{2-d+\sqrt{d^2+12d+4}}{4d}=:\si_0(d).
 \end{equation*}
We can check that $\si_0(d)>1/d$, and the gap between this value
 and the threshold for long range scattering, $1/d$, is described as
 follows, in dimensions $d=1$ and $2$ (for $d\ge 3$, the condition
 \eqref{eq:wave-op} for the existence of wave operators shows that the
 picture is less complete): if ($d\le 2$ and) $\si>1/d$,
 \begin{itemize}
  \item The wave operators (mapping $u_-$
   to $u_0$) are well-defined from $\Sigma$ to $\Sigma$.
 \item In the small data case, scattering is known in $\Sigma$: if
   $u_0\in \Sigma$ is such that $\|u_0\|_\Sigma$ is 
   sufficiently small, then there exists $u_+\in \Sigma$ such that
   \begin{equation}\label{eq:AC-Sigma}
   \|U_0(-t)u(t)-u_+\|_\Sigma\Tend t {+\infty}0.
 \end{equation}
\item Some partial asymptotic completeness holds (\cite{BGTV23}): for any
   $u_0\in\Sigma$, there exists a unique asymptotic state $u_+\in H^1(\R^d)$,
   such that
   \begin{equation}\label{eq:cvH1-23}
   \|U_0(-t)u(t)-u_+\|_{H^1(\R^d)}= \|u(t)-U_0(t)u_+\|_{H^1(\R^d)}\Tend t {+\infty}0.
 \end{equation}
\end{itemize}
 A natural question in view of these two properties is: if
 $u_0\in\Sigma$ is large, does $u_+$ belong to $\Sigma$? If yes, does
 the convergence hold in $\Sigma$, in the sense that
 \eqref{eq:AC-Sigma} holds?
A first issue when addressing such questions numerically is that the
above notion of smallness  is implicit. So we have to
gradually increase the size of $u_0$ and check whether the convergence
still holds in $\Sigma$ up to the final time $\pi/2$ (in terms of
$v$).

\smallbreak

For a more quantitative statement, resume the operator present
in \eqref{eq:evol1} and \eqref{eq:evol2}, and apply them to the
identity \eqref{eq:lens}:
\begin{align*}
  & x v(t,x)\sin t -i \cos t\nabla  v(t,x) = \frac{-ie^{-i\frac{|x|^2}{2}\tan t}}{(\cos
    t)^{d/2}}\nabla u \(\tan t,\frac{x}{\cos
    t}\),\\
  & x v(t,x)\cos t +i \sin t\nabla  v(t,x) = \frac{e^{-i\frac{|x|^2}{2}\tan t}}{(\cos
    t)^{d/2}}\Big(\frac{x}{\cos t} u \(\tan t,\frac{x}{\cos
    t}\) \\
 & \phantom{x v(t,x)\cos t +i \sin t\nabla  v(t,x) =
   \frac{e^{-i\frac{|x|^2}{2}\tan t}}{(\cos 
    t)^{d/2}}}\quad + i\tan t \nabla u \(\tan t,\frac{x}{\cos
    t}\)\Big).
\end{align*}
Taking the $L^2$-norms yields
\begin{align*}
&  \|x v(t)\sin t -i \cos t\nabla  v(t)\|_{L^2} = \|\nabla u\(\tan
  t\)\|_{L^2},\\
 & \| x v(t)\cos t +i \sin t\nabla  v(t)\|_{L^2} = \|y u(\tan t)+
   i\tan t\nabla u(\tan t)\|_{L^2} .
\end{align*}
In the last term, we recognise the operator $J(\tan t)$, and the
identities listed in Lemma~\ref{lem:disp-lin} imply
\begin{equation}\label{eqn:expressionJu_using_v}
  \| x v(t)\cos t +i \sin t\nabla  v(t)\|_{L^2} =\|J(\tan t )u(\tan
  t)\|_{L^2} = \|xU_0(-\tan t )u(\tan t)\|_{L^2}. 
\end{equation}
In view of \eqref{eq:cvH1-23},
\begin{align*}
  \|x v(t)\sin t -i \cos t\nabla  v(t)\|_{L^2}
  &= \|\nabla u\(\tan
  t\)\|_{L^2}= \|U_0(-\tan t)\nabla u\(\tan
  t\)\|_{L^2}\\
 & = \|\nabla U_0(-\tan  t) u\(\tan
  t\)\|_{L^2}\Tend t {\frac{\pi}{2}}\|\nabla u_+\|_{L^2}.
\end{align*}
Therefore, if $\|x v(t)\sin t -i \cos t\nabla  v(t)\|_{L^2}$ is
unbounded as $t\to \frac{\pi}{2}$, then \eqref{eq:AC-Sigma} cannot
hold. Finally, in view of the identity
\begin{equation*}
  \begin{pmatrix}
    x v(t)\sin t -i \cos t\nabla  v(t)\\
    x v(t)\cos t +i \sin t\nabla  v(t)
  \end{pmatrix}
  =
  \begin{pmatrix}
    \sin t & -\cos t\\
    \cos t & \sin t
  \end{pmatrix}
  \begin{pmatrix}
    x v(t)\\
    i\nabla v(t)
  \end{pmatrix},
\end{equation*}
and since the determinant of the matrix on the right hand side is one,
the unboundedness of $x v(t)\sin t -i \cos t\nabla  v(t)$ in $L^2$ is
equivalent to the unboundedness of $ v(t)$ in $\Sigma$.
\smallbreak

\begin{figure}
	[h!]
	\begin{subfigure}{0.95\textwidth}
		\centering
		\includegraphics[width=0.97\textwidth]{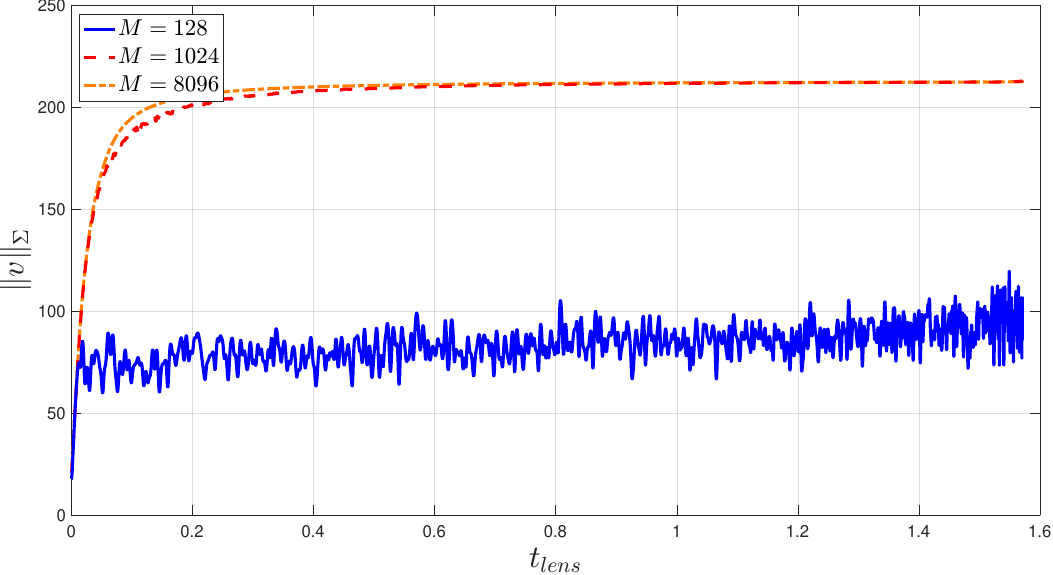}
		\caption{$\sigma=1.5$.}
	\end{subfigure}\\
	\begin{subfigure}{0.95\textwidth}
		\centering
		\includegraphics[width=0.97\textwidth]{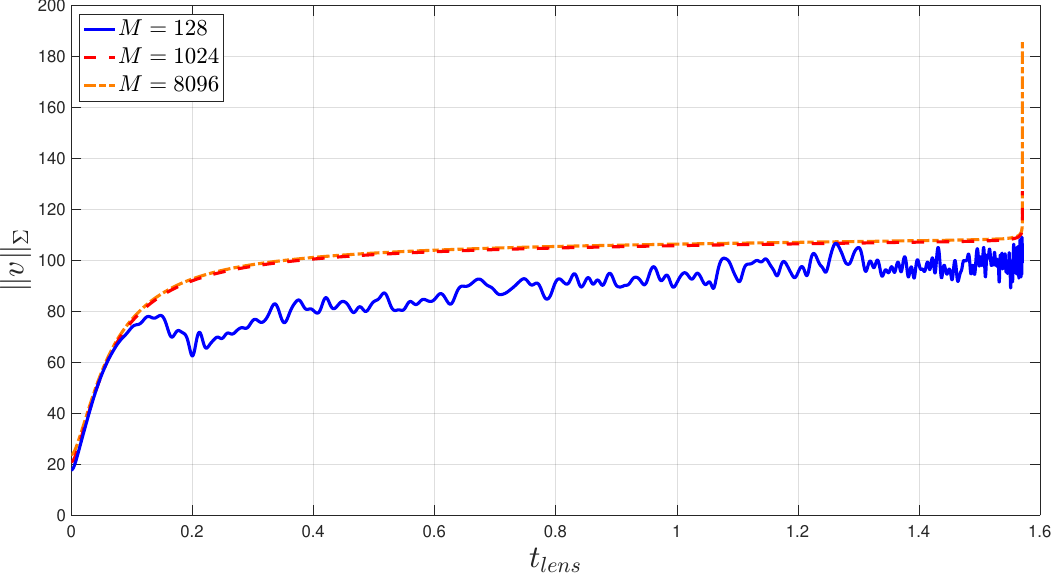}
	\caption{$\sigma=1.15$.}\label{fig:defocusing_sigma_1-15}
	\end{subfigure}
	\caption{The time-evolution of the $\Sigma$-norm of the lens-transformed variable $v$ with $\tau_{\mathrm{lens}}=2^{-14}, \|v_0\|_{\infty}=10$.}
	\label{fig:defocusing_sigma}
\end{figure}

\smallbreak

%\newpage
We consider the one-dimensional case $d=1$, where
\begin{equation*}
	\si_0(1) = \frac{1+\sqrt{17}}{4}\approx 1.28.
\end{equation*}
Similarly to Section~\ref{sec:visualisation_of_scattering_operator}, the following numerical experiment is performed with randomised initial data of the form
\begin{align}\label{eqn:initial_conditions_long_range}
	{v^M(t_{\mathrm{lens}}=0)}=\sum_{m=0}^{M-1}\lambda_m^{-1}U_m \mathcal{H}_m(x),
\end{align}
where $U_m\sim U([0,1]+i[0,1])$ and $\lambda_m$ are as in \eqref{eqn:details_sigma_norm_in_hermite_basis}. We then evolve this initial data according to \eqref{eq:nlsharmo-gen} using \ref{list:ouralgo2}.

As a reference case, we consider $\si=1.5$: we know that scattering
holds in $\Sigma$ for arbitrary large data. For this value, the cosine
factor is nontrivial in \eqref{eq:nlsharmo-gen}, so checking that
indeed $v(t,\cdot)$ always converges in $\Sigma$ as $t\to \pi/2$ is a
way to verify that the numerical scheme is robust, as illustrated by
the top picture in Figure~\ref{fig:defocusing_sigma}.

On the other hand, the value $\si=1.15$ belongs to the ``unknown
range'' $(1,\si_0(1))$. The numerical simulation for this value of
$\si$ is reported in the second picture in
Figure~\ref{fig:defocusing_sigma}. We observe that the $\Sigma$-norm
of $v$ becomes unbounded as $t\to \pi/2$ (this is to be compared with the top
picture, where the norm remains bounded). This suggests that for this
value of $\si$, asymptotic completeness in $\Sigma$ (like stated in
\eqref{eq:AC-Sigma}) does not hold for large data. In view of the
result from \cite{BGTV23}, this means that the failure in
\eqref{eq:AC-Sigma} comes from the absence of convergence of the
moment.

If indeed the scattering map may fail to map $\Sigma$ to $\Sigma$ for
$\si$ smaller than the Strauss exponent $\si_0(d)$, several questions
appear.  First, the mechanism for this failure is unclear, and should
be characterized more precisely on an analytical level. Second, if
another threshold is involved to characterize scattering in $\Sigma$,
what is its value? In principle, our numerical approach could help
obtain some bounds on this value, up to running many
experiments. Guessing a precise value or formula from numerics seems
unrealistic though, among other reasons because  only finitely many
scattering states can be considered.

{\begin{remark}\label{rmk:expected_value_initial_data}
		We note that 
		$$
		\begin{aligned}
			\mathbb{E}\left(\left\|\sum_{m=0}^{M-1}\lambda_m^{-1}U_m \mathcal{H}_m(x)\right\|_{\Sigma}^{2}\right) & =\mathbb{E}\left(\sum_{m=0}^{M-1} \lambda_{m}^{-1}\left|U_{m}^{\omega}\right|^{2}\right)  \propto \log M
		\end{aligned}
		$$
		Thus, the family of initial conditions \eqref{eqn:initial_conditions_long_range} should be viewed as a sequence of \emph{spectrally truncated} initial data whose $\Sigma$-norm increases
		slowly (like $\log M$) with $M$. For each fixed $M$, the data belongs to
		$\Sigma$, however, in the limit $M\to\infty$ the sequence approaches the
		threshold of $\Sigma$-regularity. The moderate growth of the $\Sigma$-norm with $M$
		is visible at the initial time $t_{\mathrm{lens}}=0$ (see
		Figure~\ref{fig:defocusing_sigma}), but the observed blow-up for $\sigma=1.15$ occurs only as
		$t_{\mathrm{lens}}\to\pi/2$. This indicates that the divergence is not a
		numerical artifact caused by increasing $M$.
\end{remark}}

To provide further support to confirm that the rapid increase in $\Sigma$-norm as $t_{\mathrm{lens}}\rightarrow\pi/2$ in Figure~\ref{fig:defocusing_sigma_1-15} is not a numerical artefact we note that the increase is consistent with the growth:
\begin{align}\label{eqn:JuL2_growth}
	\|J(t)u(t)\|_{L^2}\lesssim (1+t)^{1-\frac{d\sigma}{2}},\quad
  t\ge 0,
\end{align}
which follows from the pseudoconformal evolution law and Gr\"onwall
lemma (see e.g. \cite[Proposition~11.1]{GinibreDEA} or
\cite[Section~7.5]{CazCourant}). Indeed, using 
\eqref{eqn:expressionJu_using_v} we can compute $\|J(t)u(t)\|_{L^2}$
for the above experiment when $\sigma=1.15$ and observe that the
observed growth of this quantity is, in fact, consistent with
\eqref{eqn:JuL2_growth}, cf. Figure~\ref{fig:JuL2}.

\begin{figure}[h!]
	\centering
	\includegraphics[width=0.97\textwidth]{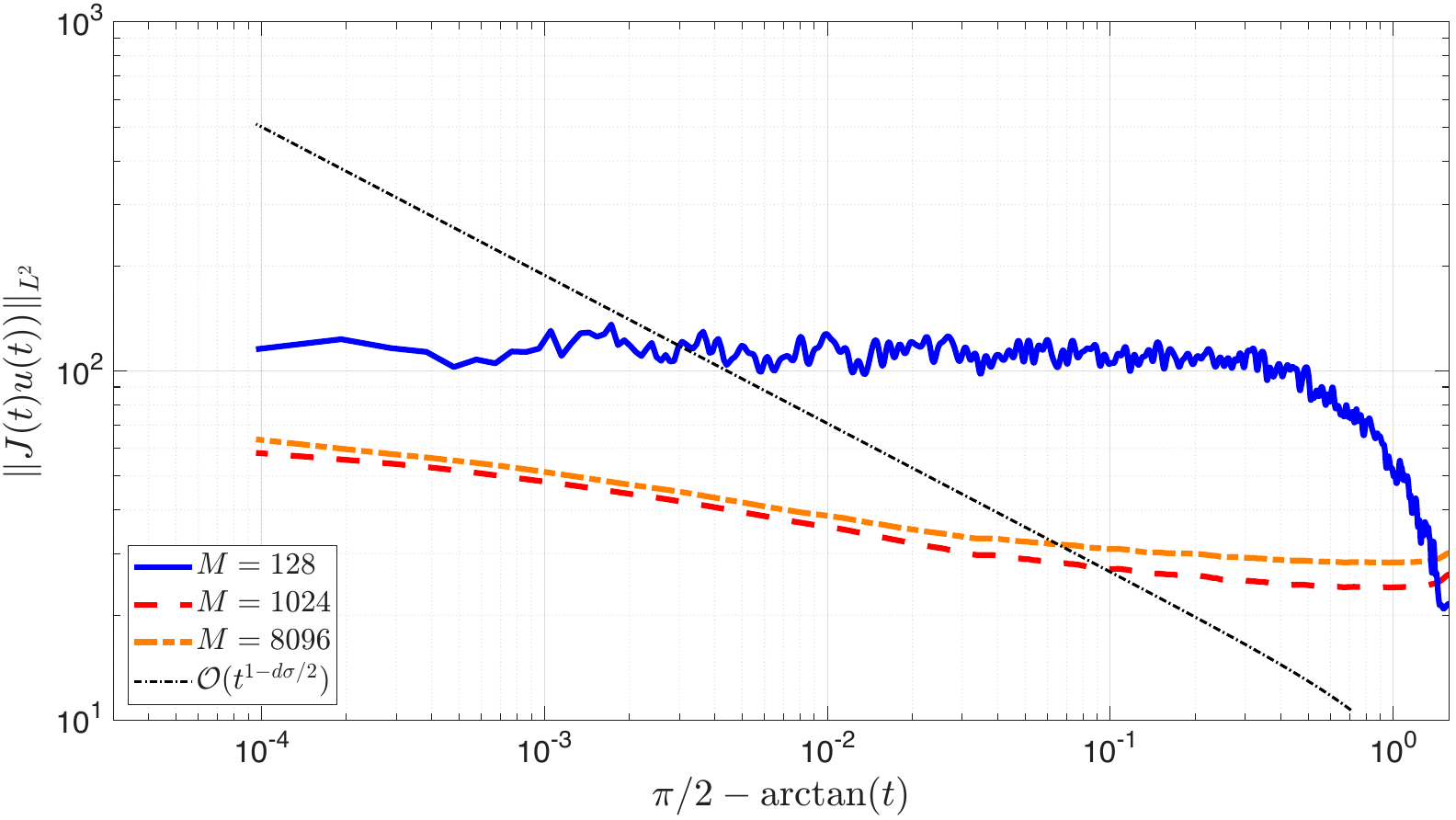}
	\caption{{The time-evolution of $\|J(t)u(t)\|_{L^2}$ for
          $d=1$, $\si=1.15$, with
          $\tau_{\mathrm{lens}}=2^{-14},\|v_0\|_{\infty}=10$, in
          log-log scale.}}
	\label{fig:JuL2}
\end{figure}

\section{Focusing long range case}
\label{sec:long-foc}

The results recall in Section~\ref{sec:long-range} turn out to be essentially
independent of the sign in front of the nonlinearity: the nonlinearity
may be defocusing or focusing, the main difference being that the
signs on the right hand side of \eqref{eq:mod-phase} are swapped in
the focusing case. 
A feature of the focusing one-dimensional cubic case,
\begin{equation}
  \label{eq:foc1D}
  i\d_t u +\frac{1}{2}\d_x^2 u = -|u|^2u,
\end{equation}
is that there exist solitary waves of the form $e^{-i\nu t}\phi_\nu(x)$. The
function $\phi_\nu$ must solve
\begin{equation*}
  \nu\phi_\nu +\frac{1}{2}\phi_\nu'' = -|\phi_\nu|^2\phi_\nu.
\end{equation*}
The ground state is the only real, positive, and even solution. It is given
by the explicit formula 
\begin{equation*}
  \phi_\nu(x) = \sqrt{2\nu}\operatorname{sech}\(x\sqrt{2\nu}\). 
\end{equation*}
Two ``natural'' Lebesgue norms are independent of $\nu>0$: $\|\phi_\nu\|_{L^1}$ and
$\|\hat \phi_\nu\|_{L^\infty}$. In \cite{CaCMP}, it is shown (after
making the normalisation for the Fourier transform consistent with the
present convention) that if 
\begin{equation*}
  \|\hat u_-\|_{L^\infty}<\frac{1}{\sqrt 2} ,
\end{equation*}
then the modified wave operator is well-defined (the smallness
assumption for the existence of modified scattering is not
explicit). This assumption was relaxed in \cite{KM-p}, $\|\hat
u_-\|_{L^\infty}<1$. 
As 
\begin{equation*}
  \int_{-\infty}^x \operatorname{sech} (y)dy = 2\arctan\(e^x\) , 
\end{equation*}
we have $\|\phi_\nu\|_{L^1}=\pi$. We also infer, since $\hat
\phi_\nu(0) =\frac{1}{\sqrt{2\pi}}\int \phi_\nu =
\frac{1}{\sqrt{2\pi}}\|\phi_\nu\|_{L^1}$,
\begin{equation*}
  \|\hat \phi_\nu\|_{L^\infty}= \sqrt{\frac{\pi}{2}}. 
\end{equation*}
It is not known whether either 
of these 
two values is sharp for the existence of the modified wave operator
(mapping $u_-$ to $u_0$ in Section~\ref{sec:long-range}) or
its inverse (mapping $u_0$ to $u_+$ in Section~\ref{sec:long-range}). The
numerical tests presented below suggest that the $L^\infty$-norm of
$\hat \phi_\nu$ is not a threshold for the existence of the inverse of
the modified wave
operator, in the sense that apparently, this operator is
not well-defined for initial data $\phi$ such that $\|\hat
\phi\|_{L^\infty}< \|\hat \phi_\nu\|_{L^\infty}$. 
\smallbreak

The parameter $\nu$ in the family of solitary waves must be chosen carefully
for the numerical simulations. If $\nu>0$ is too small, then
$\phi_\nu$ is small and diffuse. The smallness is delicate for precise
computation, the diffusion is delicate to avoid boundary effects in the
computational domain. On the other hand, is $\nu$ is too large, then
$\phi_\nu$ is large and peaked ($\phi_\nu\rightharpoonup
\pi\delta_{x=0}$ as $\nu\to \infty$), which is also delicate to manage
numerically. We have observed that the value $\nu=1/2$ is a good
compromise, and report the simulations in that case, in
Figure~\ref{fig:long-foc-Linfty}.

\begin{figure}
	[h!]
	\begin{subfigure}{0.95\textwidth}
		\centering
		\includegraphics[width=0.97\textwidth]{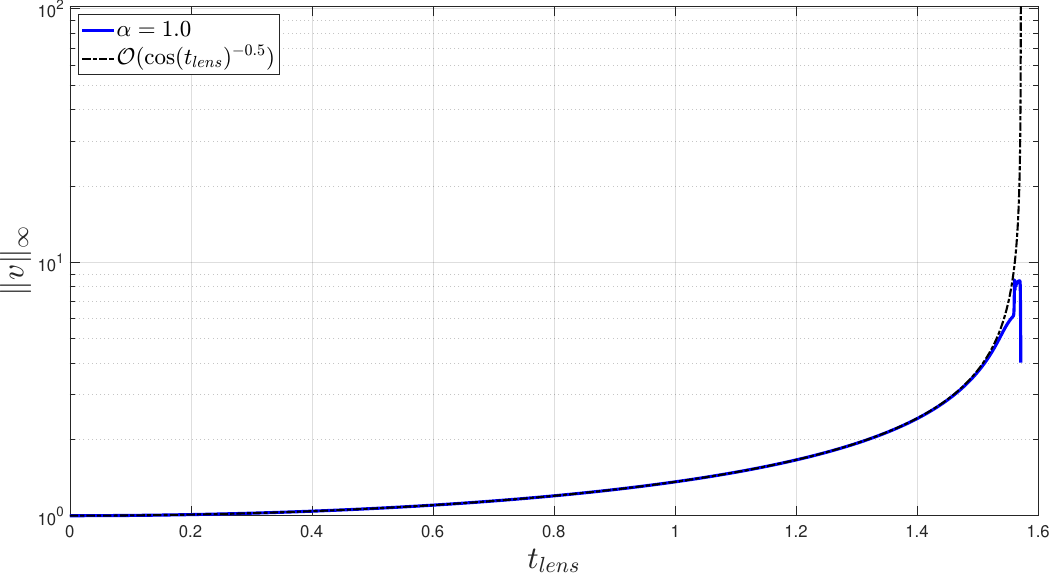}
		\caption{$\alpha=1$.}
	\end{subfigure}\\
	\begin{subfigure}{0.95\textwidth}
		\centering
		\includegraphics[width=0.97\textwidth]{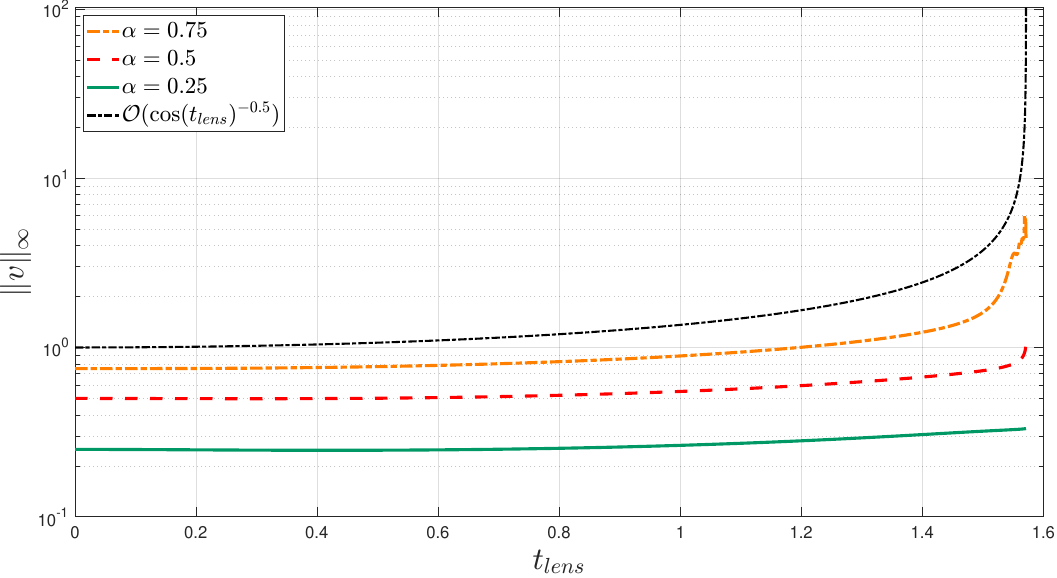}
		\caption{$\alpha\in(0,1)$.}
	\end{subfigure}
	\caption{The $L^\infty$-norm evolution for $v_0=\alpha\phi_{1/2}$ with $\tau_{\mathrm{lens}}=2^{-14}, M=8096$.}
	\label{fig:long-foc-Linfty}
\end{figure}

%\newpage 
The experiments then consist in
considering
\begin{equation*}
	u_0(x)=v_0(x)= \alpha \phi_{1/2}(x),\quad 0<\alpha\le 1.
\end{equation*}
In order to provide a good confidence level in the observations made in these graphs, we use the highest numerical resolution employed in the previous section, i.e. $\tau_{\mathrm{lens}}=2^{-14}, M=8096$. We know that for $\alpha=1$, there is no modified scattering, since
the exact solution is $u(t,x) = e^{-it/2}\phi_{1/2}(x)$, and we ask
whether or not for small values of $\alpha$, modified scattering may
be ruled out. 

We do not measure the $\Sigma$-norm
of $v$, since when there is modified scattering, the phase
modification $\Phi_+$ causes the $\Sigma$-norm of $v$ to be unbounded
as $t\to \pi/2$,  see \eqref{eq:v-long}, and stick to the observation
of the $L^\infty$-norm. When $u$ disperses like in the linear case, the
singular cosine factor $(\cos t)^{-d/2}$ in \eqref{eq:lens}
is compensated by $(\tan t)^{-d/2}$ in
\eqref{eq:lens}.
The top picture in
Figure~\ref{fig:long-foc-Linfty} shows the case
$u_0=\phi_{1/2}$, for which we know that there is no modified
scattering, and not even dispersion. We compare the behaviour of $v$ with the
function $\(\cos t\)^{-1/2}$, which is the factor in front of $u$ in
\eqref{eq:lens}. The
absence of dispersion of the solitary wave is observed by the
coincidence of the two curves, up to some time relatively close to
$\pi/2$, after which the numerical scheme becomes unstable. The second picture
illustrates, from different values of $\alpha$, the convergence or divergence of $|v|$ as time approaches $\pi/2$, with the upper bound in time for
stability  observed in the case $\alpha=1$. The growth in the case
$\alpha=0.75$ suggests that the value $\alpha=1$ may not be the threshold
to decide whether there is modified scattering or not, and modified
scattering may be ruled out for initial data smaller than the ground
state. 

%\begin{figure}
%	[h!]
%	\begin{subfigure}{0.95\textwidth}
%		\centering
%		\includegraphics[width=0.97\textwidth]{images/focusing_Sigmanorm_sigma_1_alpha_1-0.pdf}
%		\caption{$\alpha=1$.}
%	\end{subfigure}\\
%	\begin{subfigure}{0.95\textwidth}
%		\centering
%		\includegraphics[width=0.97\textwidth]{images/focusing_Sigmanorm_multiple_sigma_1_alpha0-5.pdf}
%		\caption{$\alpha\in(0,1)$.}
%	\end{subfigure}
%	\caption{\georg{[Draft figure to be updated]} $\Sigma$-Norm
%          $\tau_{\mathrm{lens}}=2^{-14}, v_0=\alpha\phi_{1/2}$, $M=8096$.}
%        \label{fig:long-foc-Sigma}
%\end{figure}

Considering $u_0= \alpha \phi_{1/2}$ with $\alpha>0$ is
reminiscent of the case considered in \cite{SatsumaYajima74}. Using
the complete integrability of the one-dimensional cubic Schr\"odinger
equation, the authors provide explicit formulas for the propagation of
multiples of solitons, and Section~5 there concerns exactly the above case,
with $\alpha$ possibly an integer, or a real number ``close'' to an
integer. The second case is the one to be compared with our
simulations: for $\alpha\in (1/2,3/2)$
(\cite[Section~5.2]{SatsumaYajima74}, case $N=1$), simulations
reported in \cite[Figure~2]{SatsumaYajima74} suggest that for
$\alpha=0.8$, the solution $u$ is not dispersive. This appears to be
consistent with our experiment in the case $\alpha=0.75$. As pointed
out in Section~\ref{sec:long-range}, more thorough comparisons of the
numerical approaches based on the lens transform and on complete
integrability should be considered.

\bigbreak

The fact that the ground state does not yield the threshold between
scattering and non-scattering in the focusing would not be new
though. Indeed, in \cite{Masaki2015,Masaki2017}, the author has proven
that for
\begin{equation*}
  \max\(\frac{1}{d},\frac{2}{d+2}\)<\si<\frac{2}{d},
\end{equation*}
initial data which are smaller that the ground state do not behave
like the linear Schr\"odinger flow in the large time limit. Another
important aspect to be considered is the topology in which the
scattering is addressed: above, we have measured the $L^\infty$ norm
of the Fourier transform of $\phi$, while in
\cite{Masaki2015,Masaki2017}, and in e.g. \cite{KMMV17,KMMV19}, it is
shown that for such $\si$, a priori bounds for the solution in
a suitable function space imply scattering.  Finding analogues for
such statements in the case $\si=d=1$ is a challenging question.

\section{Conclusions}
\label{sec:conclusions}

We have recalled several known features of the scattering operator
associated to the nonlinear Schr\"odinger equation \eqref{eq:nls}, and
established two new conservation laws. These analytical properties have
served as tests to validate the numerical approach presented here,
which relies on the lens transform. This transform avoids spatial
boundary effects since the harmonic oscillator is confining. It also
avoids the issue of computing over large time and accumulating numerical
errors, as time is compactified thus allowing for efficient, reliable simulations. By using Hermite functions (which
form an eigenbasis of the harmonic oscillator) and Lie time splitting,
our simulations are in very good agreement with known properties of
the scattering operator, even when this operator has to be modified,
due to long range effects (only the one-dimensional cubic case is
considered here).
\smallbreak

This confidence in the numerical approach led us to consider some
questions remaining open on the analytical level at the present
time. Our simulations suggest the following conjectures:
\begin{itemize}
	\item In the $L^2$-supercritical r\'egime $\si>2/d$, there may
          be no rotating point for the scattering operator.
        \item In the intermediate range $1/d<\si<\si_0(d)$, scattering
          may not hold in $\Sigma$ for large data.  
	\item In the focusing cubic one-dimensional case, the ground
          state(s) may not be the minimal  object to rule out the
          existence of modified scattering. 
\end{itemize}
Apart from securing these conjectures  with more 
numerical experiments, possibly with different approaches, some next
natural questions could be:
\begin{itemize}
\item In the intermediate range $1/d<\si<\si_0(d)$, for large data, is there
  convergence of intermediate momenta in \eqref{eq:AC-Sigma}, that is,
  is there $s\in (0,1)$ such that $x\mapsto \<x\>^su_+(x)\in
  L^2(\R^d)$ and
  \begin{equation*}
    \left\|\<x\>^s\(U_0(-t)u(t,x) - u_+(x)\)\right\|_{L^2}\Tend t
    {+\infty} 0?
  \end{equation*}
\item In the focusing cubic one-dimensional case, which criterion
  regarding the initial state $u_0$ 
  makes it possible to decide whether there is modified scattering or not?
\end{itemize}

\appendix

\section{Fast, stable Hermite transforms}\label{app:hermite_transform}
As part of the splitting method to solve \eqref{eq:nlsharmo-gen}, we
have to be able to efficiently transform between function values and
Hermite coefficients. In particular, let
$\mathbf{x}=(x_0,\dots,x_{M-1})$ be the Gauss--Hermite quadrature
nodes on $\R$ of degree $M$. For a function $f(x)$ we can thus
approximate its Hermite expansion coefficients as follows: Let
$f(x)=\sum_{m\ge 0}c_m \mathcal{H}_{m}(x)$, then 
\begin{align*}
	c_m=\int_{\R} f(x) \mathcal{H}_m(x) dx=\int_{\R} f(x) \mathrm{H}_m(x)e^{\frac{x^2}{2}} e^{-x^2} dx \approx \sum_{k=0}^{M-1}w_kf(x_k)\mathrm{H}_m(x_k)e^{\frac{x_k^2}{2}},
\end{align*}
where $w_k$ are the Gauss--Hermite quadrature weights. In other words the transformation \textit{from function values to Hermite coefficients} can be expressed as the following matrix multiplication
\begin{align*}
	\mathbf{c}=\mathsf{T}f(\mathbf{x}),
\end{align*}
where $\mathsf{T}_{km}=w_m\mathrm{H}_k(x_m)\exp(x_m^2/2).$ The entries
of $\mathsf{T}$ grow super-exponentially in the number of Hermite
modes $M$, and as a result both the assembly of the matrix and the
computation with this matrix become unstable even for moderate values
of $M$. This was realised in \cite{bunck09}, and further developed in
\cite{maierhoferwebb25} (see also the Julia implementations in
\cite{QuantumTimeSteppers,FastGaussQuadrature,FastTransforms}), who
proposed a stable way of computing the entries of $\mathsf{T}$ (see
\cite[Appendix]{bunck09}), and a stable way of computing with
$\mathsf{T}$ (see \cite{maierhoferwebb25}). %For completeness we outline the specific algorithm which we used in our simulations below. In particular, instead of working with $\mathsf{T}$ directly, we will construct an orthogonal matrix $\mathsf{Q}$ and a diagonal matrix $\mathsf{D}$ such that 
%\begin{align*}
%	\mathsf{T}=\mathsf{D}\mathsf{Q}.
%\end{align*}
The basis for this stable construction is the following recurrence for the normalised Hermite polynomials
\begin{align*}
	\mathrm{H}_{k+1}(x)=\sqrt{\frac{2}{k+1}}\, x\, \mathrm{H}_k(x) - \sqrt{\frac{k}{k+1}}\, \mathrm{H}_{k-1}(x),
\end{align*}
and the reader is referred to \cite{maierhoferwebb25} for full details on the algorithm used in our implementation.
%We then construct the two matrices $\mathsf{D},\mathsf{Q}$ by recursively filling the rows of $\mathsf{Q}=(\mathbf{q}_0\ \mathbf{q}_1\ \cdots\ \mathbf{q}_{M-1})^{\top}$. In the following we denote by $\odot$ the entrywise product of two vectors.
%\begin{enumerate}
%	\setcounter{enumi}{-1}
%	\item Initialise $\mathbf{p}_0= \pi^{-1/4}, \mathbf{p}_1=2^{1/2}\pi^{-1/4} \mathbf{x}$, $\mathbf{s}_1=(1,\dots,1)^{\top}$.
%	\item Compute $\tilde{\mathbf{p}}_{k+1}=\sqrt{\frac{2}{k+1}} \mathbf{x}\odot\mathbf{p}_k-\sqrt{\frac{k}{k+1}}\mathbf{s}_{k}\odot\mathbf{p}_{k-1}$
%	\item Rescale $\mathbf{p}_{k+1}=\mathbf{s}_{k+1}\odot \mathbf{p}_k$ to avoid numerical overflow, where
%	\begin{align*}
%		(\mathbf{s}_{k+1})_m=\begin{cases}
%			1,&\text{if}\ (\tilde{\mathbf{p}}_{k+1})_m< 100,\\
%			|(\tilde{\mathbf{p}}_{k+1})_m|^{-1},& \text{otherwise.}
%		\end{cases}
%	\end{align*}
%	\item Repeat (1)-(2) until $k=M-2$.
%	\item Normalise each row of $\mathsf{Q}$ and define the diagonal matrix $\mathsf{D}$:
%	\begin{align*}
%		\mathbf{q}_k&=(\mathbf{p}_k\odot|\mathbf{p}_{M-1}|^{-1})\frac{1}{\sqrt{M-1}}\bigodot_{j=k+1}^{M-1}\mathbf{s}_j,\\
%		\mathsf{D}_{mm}&=\prod_{j=0}^{M-1}(\mathbf{s}_j)_m\exp\left(\frac{x_m^2}{2}\right)\frac{1}{\sqrt{M-1}|(\mathbf{p}_{M-1})_m|},
%	\end{align*}
%	for $k=0,\dots, M-1$.
%\end{enumerate}

\section{Auxiliary identities for Hermite expansions}\label{app:hermit_identities}
Here we recall a few identities for Hermite expansions which are used in the efficient implementation of our algorithm for computing the lens transform (cf. Section~\ref{sec:methodology}). For the following identities, we suppose that we have a representation of a function $f$ in the form
\begin{align*}
	f(x)=\sum_{\substack{m\in\N^d\\0\le m\le M-1}} \alpha_m\mathcal{H}_m(x).
\end{align*}
\subsection{Fourier transform}
The Fourier transform of $f$ can then be written as
\begin{align*}
	\mathcal{F}f(\xi)&=\sum_{\substack{m\in\N^d\\0\le m\le M-1}} \alpha_m\mathcal{F}\mathcal{H}_m(\xi)\\
	&=\sum_{\substack{m\in\N^d\\0\le m\le M-1}} \alpha_m(-i)^m\mathcal{H}_m(\xi),
\end{align*}
where we used the well-known expression for the Fourier transform of Hermite functions (cf. \cite[18.17.21\_3]{NIST:DLMF}).
\subsection{Differentiation}
In the interest of notational simplicity, we restrict the following to
the one dimensional case. The case $d\ge 2$ can be treated similarly by
tensor product extension. In order to compute the derivative of a
Hermite expansion, we note that the normalised Hermite polynomials are
of the form 
\begin{align*}
\mathrm{H}_{m}=\frac{1}{\sqrt{2^m\sqrt{\pi}m!}} \widetilde{\mathrm{H}}_m,
\end{align*}
where $\widetilde{\mathrm{H}}_m$ satisfies the identity
\begin{equation*}
  \frac{d}{dx}\(\widetilde{\mathrm{H}}_{m}e^{-x^2/2}\)=
  x\widetilde{\mathrm{H}}_{m}
  e^{-x^2/2}-\widetilde{\mathrm{H}}_{m+1}e^{-x^2/2},\quad  
  m\ge 0,
\end{equation*}
cf. \cite[18.9.26]{NIST:DLMF}. This leads to the identity
\begin{align*}
	f'(x)=\sum_{m\ge 0}\beta_m \mathrm{H}_m e^{-x^2/2},
\end{align*}
where $\beta_m=\gamma_m-\sqrt{2(m+1)}\alpha_{m+1}$ and $\gamma_m$ are the Hermite coefficients of $xf(x)$, i.e.
\begin{align*}
	xf(x)=\sum_{m\ge 0}\gamma_m \mathrm{H}_m e^{-x^2/2}.
\end{align*}

\section{Continuity of scattering operator with respect to
          $\sigma$}\label{app:continuity_argument_sigma} 

 In this section, we sketch an analytic explanation for
  Figure~\ref{fig:rotating_points_supercritical_regime}, in the case
  $d=1$, where we consider a slightly $L^2$-supercritical
  nonlinearity, $\si=2+\eps$, $0<\eps\ll 1$. Consider the solutions
  $v$ and $v^\eps$ corresponding to the $L^2$-critical case and this
  slightly $L^2$-supercritical case, respectively, with  possibly different data
  at $-\frac{\pi}{2}$,
\begin{equation}
\label{eq:cont-si}
\left\{
  \begin{aligned}
    &i\d_t v + \frac{1}{2}\d_x^2 v = \frac{x^2}{2}v + |v|^4v,\quad
      v_{\mid t=-\pi/2}=v_-\in \Sigma,\\
    & i\d_t v^\eps + \frac{1}{2}\d_x^2 v^\eps = \frac{x^2}{2}v^\eps +
      (\cos t)^\eps |v^\eps|^{4+2\eps}v^\eps,\quad v^\eps_{\mid
      t=-\pi/2} = v^\eps_-\in \Sigma.
  \end{aligned}
\right.
\end{equation}
Since $v^\eps_-$ will eventually be close to $v_-$ ($\not\equiv 0$), we assume
\begin{equation*}
  \sup_{\eps\in (0,1)} \|v^\eps_-\|_{\Sigma} \le 2 \|v_-\|_\Sigma.
\end{equation*}
As the nonlinearity is defocusing, we infer a uniform bound of the form
  \begin{equation*}
    \sup_{-\pi/2\le t\le \pi/2}\|v(t)\|_{\Sigma} +
  \sup_{\eps\in (0,1)}   \sup_{-\pi/2\le t\le
    \pi/2}\|v^\eps(t)\|_{\Sigma}\le C\(\|v_-\|_\Sigma\). 
  \end{equation*}
Since we are in the one-dimensional case, $H^1(\R)\hookrightarrow
  L^\infty(\R)$, and we deduce that there exists $M\ge 1$ (this lower
  bound is imposed to simplify the formulas below) such that
\begin{equation}\label{eq:M}
    \sup_{-\pi/2\le t\le \pi/2}\|v(t)\|_{L^\infty} +
  \sup_{\eps\in (0,1)}   \sup_{-\pi/2\le t\le
    \pi/2}\|v^\eps(t)\|_{L^\infty}\le M. 
  \end{equation}
We readily observe that the difference $w=v-v^\eps$ solves
  \begin{align*}
   i\d_t w+ \frac{1}{2}\d_x^2 w = \frac{x^2}{2}w +|v|^4v-(\cos t)^\eps
    |v^\eps|^{4+2\eps}v^\eps\quad ;\quad w_{\mid t=-\pi/2}=v_--v_-^\eps.
  \end{align*}
  Inserting the term $\pm (\cos t)^\eps|v|^{4+2\eps}v$,  energy
  estimate yields, for $t\in
  \left[-\frac{\pi}{2},\frac{\pi}{2}\right]$,
   \begin{align*}
    \|w(t)\|_{L^2} &\le \|w(-\pi/2)\|_{L^2}+
2 \int_{-\pi/2}^t (\cos s)^\eps \left\| |v|^{4+2\eps}v -
      |v^\eps|^{4+2\eps}v^\eps\right\|_{L^2}ds \\
&\quad +2\int_{-\pi/2}^t
    \left\||v|^4v-  (\cos s)^\eps |v|^{4+2\eps}v \right\|_{L^2}ds,
  \end{align*}
where the implicit time variable for $v$ and $v^\eps$ is $s$.
  In view of \eqref{eq:M}, for all $s\in
  \left[-\frac{\pi}{2},\frac{\pi}{2}\right]$,
  \begin{equation*}
   (\cos s)^\eps \left\| |v|^{4+2\eps}v -
      |v^\eps|^{4+2\eps}v^\eps\right\|_{L^2}\le
    CM^{4+2\eps}\|w(s)\|_{L^2}\le C M^6 \|w(s)\|_{L^2},
  \end{equation*}
for some constant $C$ uniform in $\eps\in (0,1)$. In order to apply
Gr\"onwall's lemma, we therefore focus on the source term
\begin{equation*}
  |v|^4v-  (\cos s)^\eps |v|^{4+2\eps}v = |v|^4v-  
  |v|^{4+2\eps}v + \(1 - (\cos s)^\eps\) |v|^{4+2\eps}v .
\end{equation*}
With $y\ge 0$  a placeholder for $|v|^4$, and up to changing $\eps$
with $2\eps$ to lighten notations, consider the function
\begin{equation*}
  f^\eps(y)=y-y^{1+\eps}.
\end{equation*}
This function is increasing on
$\left[0,\frac{1}{(1+\eps)^{1/\eps}}\right]$, decreasing on
  $\(\frac{1}{(1+\eps)^{1/\eps}}, +\infty\)$, and so
  \begin{equation*}
    \sup_{0\le y\le M^4}|f^\eps(y)|\le \max\( f^\eps\(
    \frac{1}{(1+\eps)^{1/\eps}}\),|f^\eps(M^4)|\),
  \end{equation*}
where we have chosen for $M$ the constant from \eqref{eq:M}. We
compute
\begin{equation*}
 f^\eps\(\frac{1}{(1+\eps)^{1/\eps}}\)=
 \frac{1}{(1+\eps)^{1/\eps}}\(1- \frac{1}{1+\eps}\) = \O(\eps),
\end{equation*}
and (since $M\ge 1$)
\begin{equation*}
  |f^\eps(M^4)|= M^{4+4\eps}-M^4 = M^4\(e^{4\eps\log M}-1\)=\O(\eps). 
\end{equation*}
We infer from \eqref{eq:M} and the conservation of the $L^2$-norm for
$v$,
\begin{align*}
  \int_{-\pi/2}^{\pi/2} \left\||v|^4v-|v|^{4+2\eps}v\right\|_{L^2}ds
& \le 
  \int_{-\pi/2}^{\pi/2}
  \left\||v|^4-|v|^{4+2\eps}\right\|_{L^\infty}\|v(s)\|_{L^2}ds \\
& \le \pi \|v_-\|_{L^2}\sup_{0\le y\le M^4} |f^{\eps/2}(y)| \le C\eps.
\end{align*}
The last part of the source term is
\begin{equation*}
  \int_{-\pi/2}^t \(1-\(\cos s\)^\eps\) \left\|
    |v|^{4+2\eps}v\right\|_{L^2}ds \le M^6 \|v_-\|_{L^2}
 \int_{-\pi/2}^{\pi/2} \(1-\(\cos s\)^\eps\) ds.
\end{equation*}
Taylor formula for the function $g(\eps) = (\cos s)^\eps$ yields, for
$s\in \(-\frac{\pi}{2},\frac{\pi}{2}\)$,
\begin{equation*}
  1- (\cos s)^\eps= g(0)-g(\eps) = -\eps \log \cos s\int_0^1 \(\cos
  s\)^{\theta\eps}d\theta \le \eps \log\frac{1}{\cos s}.
\end{equation*}
Since the right hand side is integrable on
$\(-\frac{\pi}{2},\frac{\pi}{2}\)$, we conclude
\begin{equation*}
  \int_{-\pi/2}^{\pi/2}
    \left\||v|^4v-  (\cos s)^\eps |v|^{4+2\eps}v \right\|_{L^2}ds\le C\eps,
\end{equation*}
for some constant $C$ independent of $\eps$. Invoking Gr\"onwall's lemma,
we have proved:
  \begin{proposition}\label{prop:continuity}
  There exists $C$ independent of $\eps\in (0,1)$ such that the
  solutions to \eqref{eq:cont-si} satisfy
  \begin{equation*}
    \sup_{-\pi/2\le t\le \pi/2}\|v(t)-v^\eps(t)\|_{L^2}\le
    C\|v_--v^\eps_-\|_{L^2}+ C\eps.
  \end{equation*}
  \end{proposition}
This result explains why the existence of rotating points for the
scattering operator when the nonlinearity is not $L^2$-critical is
doubtful in view of our numerical experiments. More precisely, in
Figure~\ref{fig:rotating_points_supercritical_regime}, we consider
$\si=2.01$, hence $\eps = 10^{-2}$, which corresponds to the order of
magnitude in the error observed, compared to the case $\si=2$ where
the existence of rotating points is granted analytically and validated
numerically.

\bibliographystyle{abbrv}
\bibliography{scatt}

\end{document}